\documentclass[12 pt,reqno]{amsart}

\usepackage{latexsym, amsmath, amssymb, longtable, booktabs,amscd,microtype,booktabs,cases,mathrsfs}
\usepackage{euscript}
\usepackage[square,numbers]{natbib}
\usepackage{graphicx}
\usepackage[dvipsnames]{xcolor}
\usepackage{caption}
\usepackage{dsfont}
\usepackage[all]{xy}
\usepackage{enumerate}
\usepackage{upgreek}
\usepackage{bm}
\usepackage{relsize}
\usepackage{multirow}
\usepackage{tikz}
\usetikzlibrary{matrix,arrows,decorations.pathmorphing}
\usepackage{collectbox}
\usepackage{enumerate}
\usepackage[colorinlistoftodos,prependcaption, textsize=tiny]{todonotes}
\reversemarginpar
\usepackage{amsmath}
\usepackage{enumitem}
\usepackage{hyperref}
\usepackage{cleveref}
\let\oldtocsection=\tocsection

\let\oldtocsubsection=\tocsubsection

\let\oldtocsubsubsection=\tocsubsubsection

\renewcommand{\tocsection}[2]{\hspace{0em}{\vspace{0.5em}}\oldtocsection{#1}{#2}}
\renewcommand{\tocsubsection}[2]{\hspace{1em}{\vspace{0.5em}}\oldtocsubsection{#1}{#2}}
\renewcommand{\tocsubsubsection}[2]{\hspace{2em}\oldtocsubsubsection{#1}{#2}}
\hypersetup{
	colorlinks,
	citecolor=red,
	filecolor=black,
	linkcolor=blue,
	urlcolor=black
}
\numberwithin{equation}{section}
\setlist[enumerate]{font=\upshape}

\newcommand{\Z}{\mathbb{Z}}

\newcommand{\D}{\EuScript{D}_n}

\newcommand{\La}{\EuScript{G}}

\newcommand{\al}{\alpha}

\newcommand{\lm}{\underbar{\textbf{m}}}

\newcommand{\C}{\mathbb{C}}

\newcommand\N{\mathbb{N}}

\def\gg{{\mathfrak{g}}}

\newcommand{\Q}{\mathbb{Q}}

\makeatother

\newtheorem{thm}{Theorem}[section]
\newtheorem{theorem}[thm]{Theorem}
\newtheorem{cor}[thm]{Corollary}

\newtheorem{prop}[thm]{Proposition}
\newtheorem{lemma}[thm]{Lemma}
\theoremstyle{definition}
\newtheorem{definition}[thm]{Definition}
\newtheorem{remark}[thm]{Remark}

\theoremstyle{definition}

\theoremstyle{remark}

\theoremstyle{remark}

\makeatletter
\def\imod#1{\allowbreak\mkern10mu({\operator@font mod}\,\,#1)}
\makeatother
\title[Classification of Harish-Chandra modules ]{\textbf{Classification of irreducible Harish-Chandra modules over extended Divergence zero Lie algebras}}
\author[Sudipta Mukherjee]{{Sudipta Mukherjee\vspace{0.1cm}}}
\begin{document}
	\maketitle
	
\begin{abstract}
	Let $\mathcal{A}_n = \C[t_1^{\pm1}, t_2^{\pm1}, \ldots, t_n^{\pm1}]$, and let $\D$ denote the divergence-zero subalgebra of $\text{Der}\,(\mathcal{A}_n)$. In this paper, we classify irreducible Harish-Chandra modules over the extended divergence-zero Lie algebra $\La:=\D \ltimes \mathcal{A}_n$ with nontrivial $\mathcal{A}_n'$-action, where $\mathcal{A}'_n= \oplus_{{\bf{m}} \in \Z^n\setminus \{\bf{0}\}} \C t^{\bf{m}}$. We prove that every such module is either cuspidal or a generalized highest weight module.  We further prove that every irreducible generalized highest weight $\La$-module is an irreducible highest weight module with respect to a suitable triangular decomposition of $\La$. As a consequence, we obtain a classification of irreducible Harish-Chandra modules over $\La$ with nontrivial $\mathcal{A}_n'$-action.
\end{abstract}
\tableofcontents
		\section{Introduction:} 
		
		Infinite dimensional Lie algebras and their representations have been studied extensively because of their importance in both mathematics and mathematical physics. In particular, they arise naturally in conformal field theory, string theory, and soliton theory. An important class of representations consists of those that decompose into finite-dimensional weight spaces. These representations are known as Harish-Chandra modules in the literature, and their classification for various classes of infinite dimensional Lie algebras has been studied by many authors.
		
		One of the most prominent and widely studied infinite-dimensional Lie algebras is the Virasoro algebra, which plays a fundamental role in theoretical physics. It can be realized as the one dimensional central extension of the Lie algebra of diffeomorphisms of the circle, or equivalently, as the one dimensional central extension of the Lie algebra of derivations of the Laurent polynomial ring in one variable. The classification of all irreducible Harish-Chandra modules over the Virasoro algebra was conjectured by V. Kac \cite{bom}, and later proved by O. Mathieu in 1992 \cite{om}.
		
		The higher dimensional analogue of the Virasoro algebra is the Witt algebra. Let $\mathcal{A}_n$ denote the Laurent polynomial ring in $n$ variables. Its derivation algebra $\mathcal{W}_n=\text{Der}\,(\mathcal{A}_n)$ is popularly known as the Witt algebra and can be identified with the Lie algebra of polynomial vector fields on an $n$-dimensional torus (see \Cref{a93}). In contrast to the one-variable case, $\mathcal{W}_n$ is centrally closed for $n \geq 2$. The representation theory of Witt algebras has been studied extensively by many mathematicians (see \cite{bill,rao2,r3,guo,guo11,l1,l,maz}). In \cite{sh}, Shen introduced a class of modules over the Witt algebra $\mathcal{W}_n$ arising from irreducible modules over the general linear Lie algebra $\mathfrak{gl}_n$, which were also given by Larsson in \cite{l}. Their irreducibility was determined by Eswara Rao in \cite{rao2}. Later, Billig and Futorny gave the complete classification of all irreducible Harish-Chandra modules over $\mathcal{W}_n$ \cite{bill}.
		
		Extended affine Lie algebras (EALAs for short; see \cite{a207} for the definition) were introduced by Hoegh-Krohn and Torresani in \cite{a201} under the name of quasisimple Lie algebras.  Over the last three decades, their structure theory has been intensively studied by many mathematicians (see \cite{a205,a203,a202,a204} and references therein). These algebras provide a natural framework extending finite-dimensional simple Lie algebras (nullity $0$) and affine Kac--Moody algebras (nullity $1$). Other important examples include toroidal EALAs \cite{a206}, Hamiltonian EALAs, contact EALAs, and minimal EALAs \cite{a207}. It is known that full toroidal Lie algebras are not EALAs, since they do not admit a nondegenerate symmetric invariant bilinear form. To obtain an EALA structure, one considers an important subalgebra of $\mathcal{W}_n$ consisting of all skew derivations of $\mathcal{A}_n$. This subalgebra of $\mathcal{W}_n$ is called the Lie algebra of divergence-zero vector fields on a torus, denoted by $\D$. In \cite{a208}, it was shown that the classification of irreducible integrable Harish-Chandra modules over nullity $2$ toroidal EALAs of type $A_1$ reduces to the classification of Harish-Chandra modules over $\EuScript{D}_2$. The connection of the Lie algebra $\EuScript{D}_2$ with the vertex algebra has been established in \cite{a209}. The restrictions of the tensor modules over $\mathcal{W}_n$ (introduced by Shen \cite{sh}) to the divergence-zero subalgebra $\D$ were studied in \cite{ta}, where analogous irreducibility criteria were obtained.
		
		The Witt algebra $\mathcal{W}_n$ acts naturally on the polynomial algebra $\mathcal{A}_n$ by derivations. Therefore, the emerging semidirect product
		$\bar{\EuScript{G}}_n:=\mathcal{W}_n \ltimes \mathcal{A}_n$ arises naturally and is called the extended Witt algebra. In \cite{r3}, Eswara Rao classified irreducible modules with finite-dimensional weight spaces over $\bar{\EuScript{G}}_n$ that admit compatible actions of both the Lie algebra $\mathcal{W}_n$ and $\mathcal{A}_n$. This result was extended to a classification of indecomposable modules in this category \cite{b}. Irreducible Harish-Chandra modules over $\overline{\La}_n$ with nontrivial $\mathcal{A}_n$-action were classified in \cite{guo}.
		
		In this paper, we study irreducible representations of the subalgebra $\La_n:=\D \ltimes \mathcal{A}_n$ of $\bar{\La}_n$ that have finite-dimensional weight spaces. Since the classification of irreducible Harish-Chandra modules over $\D$ is still unknown, we impose the additional assumption that the modules have nontrivial $\mathcal{A}_n'$-action, where
		$\mathcal{A}_n'=\bigoplus_{{\bf m}\in \Z^n\setminus \{{\bf 0}\}}\C t^{\bf m}$. In \cite{bta}, Billig and Talboom classified irreducible and indecomposable modules over $\La_n$ admitting compatible actions of $\D$ and $\mathcal{A}_n$. For the two variable case, representations of $\La_2$ were studied in \cite{gl,lt,ts}.
		
		Our methods are inspired by and adapted from the work of Guo-Liu \cite{gl}, Lin-Tan \cite{lt}, and Lu--Zhao \cite{lz} on higher rank Virasoro like algebras. A key difference is that higher-rank Virasoro algebras contain a solenoidal Lie algebra (see \cite{sou} for the definition), which is isomorphic to the centerless Virasoro algebra and plays a crucial role in the classification problem. In contrast, the divergence-zero Lie algebra does not contain any subalgebra isomorphic to the centerless Virasoro algebra. In particular, the techniques in \cite{guo,lz}, which rely on the representation theory of the Virasoro algebra, cannot be applied in our setting. So, the divergence-zero setting requires several additional arguments, since one must work with the subalgebra $\D$ in place of the full derivation algebra. 	We first show that any irreducible uniformly bounded $\La_n$-module with nontrivial $\mathcal{A}_n'$-action admits compatible actions of $\D$ and $\mathcal{A}_n$. This implies that such modules are jet modules (see \Cref{a93} for the definition). By applying a result of Billig and Talboom \cite{bta}, we then classify all such modules and show that they are precisely the modules of the form $V(\lambda,\boldsymbol{\alpha},c)$ (see \Cref{a93} for the definition). We next prove that every nontrivial irreducible Harish-Chandra module over $\La_n$ is either cuspidal or a generalized highest weight module (Theorem \ref{a75}). Irreducible Harish-Chandra GHW modules over $\La_2$ were classified in \cite{ts}. However, the classification of irreducible GHW modules over $\La_n$ were not known for $n \geq 3$. We prove that every irreducible generalized highest weight module over $\La_n$ is isomorphic to a highest weight module for $\La_n$ (\Cref{a28}). Combining these results, we obtain a complete classification of irreducible Harish-Chandra modules over $\La_n$ with nontrivial $\mathcal{A}_n'$-action. The paper is organized as follows.
	\subsection{Organization of the paper} The aim of this paper is to classify all irreducible Harish-Chandra modules over the extended divergence-zero Lie algebra $\La$. \Cref{a93} lays the foundation for the paper. We begin by recalling the definition and structure of the extended divergence-zero Lie algebra $\La$. We also recall some basic definitions concerning various types of weight modules over $\La$. We conclude this section by discussing an example of a class of $\La$-modules.
	
	\vspace{0.1cm}
	\Cref{a94} is devoted to the classification of irreducible cuspidal Harish-Chandra modules over $\La$. We first prove that the action of $\mathcal{A}_n$ on such modules is associative (Theorem \ref{a19}). This allows us to apply a result of Billig and Talboom to obtain the main theorem of this section (Theorem \ref{a90}).
	
	\vspace{0.1cm}
	In \Cref{a95}, we begin by introducing the notion of a generalized highest weight (GHW) module for the Lie algebra $\La$. We prove that every nontrivial irreducible Harish-Chandra module over $\La$ is either cuspidal or a GHW module (Theorem \ref{a75}). We end the section by establishing several important properties of GHW modules.
	
	\vspace{0.1cm}
	In \Cref{a96}, we classify irreducible GHW modules over $\La$. We start by defining highest weight modules over $\La$ with respect to a suitable triangular decomposition. We then prove that every irreducible GHW module over $\La$ is isomorphic to a highest weight module for $\La$. We first establish this result for $n=2$ (\Cref{a36}) and then extend it to arbitrary $n$ by using mathematical induction.
	
	\vspace{0.1cm}
	Finally, \Cref{a97} combines the results of the previous sections to obtain the main classification theorem for irreducible Harish-Chandra modules over $\La$ with nontrivial $\mathcal{A}_n'$-action (\Cref{a100}).
	\section{Preliminaries}\label{a93}
	\subsection{Notations} Throughout this paper, we consider all vector spaces, algebras, and tensor products to be over the field of complex numbers $\C$. We shall denote the set of integers, natural numbers, and real numbers by $\Z, \N, \mathbb{R}$ respectively. We write elements of $\Z^n$ and $\C^n$ in boldface. For an $n$-tuple ${\bf{m}}=(m_1,m_2,\ldots,m_n) \in \Z^n$, we denote $(m_2,\ldots,m_n) \in \Z^{n-1}$ by $\lm$. For two elements ${\bf{m}}, {\bf{k}} \in \Z^n$, we say that ${\bf{m}} \geq {\bf{k}}$ if $m_i \geq k_i$ for all $1 \leq i \leq n$. Let $(-|-)$ denote the standard inner product on $\C^n$. The universal enveloping algebra of a Lie algebra $\gg$ is denoted by $\mathcal{U}(\gg)$. Let $\{{\bf{e_1}}, {\bf{e_2}},\ldots,{\bf{e_n}} \}$ denote the standard basis of $\C^n$. The dual space of a vector space $V$ will be denoted by $V^*$. The notation $A^{\times}$ stands for the set $A \setminus \{0\}$.
	\subsection{Basics of Extended Divergence Zero Lie Algebra}
	For a positive integer $n \geq 3$, let $\mathcal{A}_n = \C[t_1^{\pm1}, t_2^{\pm1}, \ldots, t_n^{\pm1}]$ denote the Laurent polynomial ring in $n$ variables $t_1, t_2, \ldots, t_n$. For an $n$-tuple ${\bf m} = (m_1, m_2, \ldots, m_n) \in \Z^n$, we write $t^{\bf m} = t_1^{m_1} t_2^{m_2} \cdots t_n^{m_n} \in \mathcal{A}_n$
	for a typical element in $\mathcal{A}_n$.  Let $\mathcal{W}_n$ denote the derivation algebra of $\mathcal{A}_n$.  For $1 \leq j \leq n$, define $d_j = t_j \frac{\partial}{\partial t_j}$, which acts on $\mathcal{A}_n$ as a derivation. It is well known that $\mathcal{W}_n$ is spanned by the elements
	$\{t^{\bf m}d_j: \, {\bf m}\in \Z^n,\; 1\leq j\leq n\}$. This infinite dimensional derivation algebra is a very well known classical object and is popularly known as the Witt algebra of rank $n$. The Lie bracket on $\mathcal{W}_n$ is given by
	\begin{equation*}
		[t^{\bf m} d_i, t^{\bf k} d_j]
		= k_i\, t^{{\bf m}+{\bf k}} d_j - m_j\, t^{{\bf m}+{\bf k}} d_i.
	\end{equation*}
	We now introduce an alternative set of notations for the elements of $\mathcal{W}_n$. For ${\bf p} \in \C^n$ and ${\bf m} \in \Z^n$, define
	$D({\bf p}, {\bf m}) = \sum_{i=1}^{n} p_i\, t^{\bf m} d_i$.
	With this notation, the bracket formula can be written as
	\begin{equation}\label{a45}
		[D({\bf{p}}, {\bf{m}}), D({\bf{q}}, {\bf{k}})] = D({\bf{w}},\, {\bf{m+k}}), \quad \text{where}\,\,\, {\bf{w}}=({\bf{p}}|{\bf{k}}){\bf{q}}-({\bf{q}}|{\bf{m}}){\bf{p}}.
	\end{equation}
	Geometrically, $\mathcal{W}_n$ can be interpreted as the Lie algebra of complex-valued polynomial vector fields on an $n$-dimensional torus via the identification
	$t_j = e^{i\theta_j},\,\, 1 \leq j \leq n$,
	where $\theta_j$ denotes the $j$-th angular coordinate. This has an interesting subalgebra, the Lie algebra of divergence zero vector fields, denoted by $\EuScript{D}_n$.
	\begin{prop}[\cite{ta}, Proposition 2.1]
		$\D\hspace{-0.05cm}=\hspace{-0.05cm}\{D({\bf{u}}, {\bf{r}}) \hspace{-0.05cm} \in \mathcal{W}_n: {\bf{u}} \in \C^n\hspace{-0.05cm},\, \hspace{-0.05cm}{\bf{r}} \in \Z^n,\, \hspace{-0.05cm}({\bf{u}}|{\bf{r}})=0 \}$.
	\end{prop} 
	\begin{proof}
		Let
		$X=\sum_{j=1}^{n} f_j({\bf t})\,\frac{\partial}{\partial t_j}\in \mathcal{W}_n$.
		Recall that the divergence of $X$ with respect to the coordinates $t_1,\dots,t_n$ is defined by $\text{div}\,(X)= \sum_{j=1}^{n} \frac{\partial f_j}{\partial t_j}$. Using the change of coordinates $t_j=e^{i\theta_j}$, we obtain
		\begin{equation*}
			\frac{\partial }{\partial \theta_j}=\frac{\partial t_j}{\partial \theta_j}. \frac{\partial }{\partial t_j}=it_j \frac{\partial }{\partial t_j}=id_j
		\end{equation*}
		Consequently,
		$D({\bf u},{\bf r})
		=\sum_{j=1}^{n}u_j t^{\bf r}d_j
		=-i\sum_{j=1}^{n}u_j t^{\bf r}\frac{\partial}{\partial \theta_j}$. Thus, the divergence of $D({\bf u},{\bf r})$ with respect to the angular coordinates is
		\begin{equation*}
			\text{div}\,(D({\bf{u}}, {\bf{r}}))=-i \sum_{j=1}^{n} u_j \frac{\partial t^{\bf{r}} }{\partial \theta_j}=\sum_{j=1}^{n} u_j t_j\frac{\partial t^{\bf{r}} }{\partial t_j}=\sum_{j=1}^{n} u_j r_j \,t^{\bf{r}}=({\bf{u}}| {\bf{r}})\,t^{\bf{r}}.
		\end{equation*}
		Therefore, $D({\bf{u}}, {\bf{r}}) \in \D$ if and only if $({\bf{u}}| {\bf{r}})=0$.
	\end{proof}
	It is easy to check that $\D$ is a subalgebra of $\mathcal{W}_n$. For ${\bf r}\in \Z^n$, set
	\begin{equation*}
		D_{ij}({\bf r})=r_j t^{\bf r}d_i-r_i t^{\bf r}d_j, \qquad 1\leq i,j\leq n.
	\end{equation*}
	By definition, $D_{ij}({\bf{r}})=-D_{ji}({\bf{r}}), D_{ii}({\bf{r}})=0, D_{ij}({\bf{0}})=0$. Also we have
	\begin{equation}\label{a70}
		r_kD_{ij}({\bf{r}})+r_iD_{jk}({\bf{r}})+r_jD_{ki}({\bf{r}})=0
	\end{equation}
	Moreover, it is easy to see that 
	\begin{center}
		$\D=\text{Span}\,\{D_{ij}({\bf{r}}),\, d_k: {\bf{r}} \in \Z^n,\, 1 \leq i,j,k \leq n\}$.
	\end{center} 
	A Cartan subalgebra of $\D$ is given by $\mathcal{H}=\text{Span}\,\{d_j : 1\leq j\leq n\}$. Define $\delta_i \in \mathcal{H}^{*}$ such that $\delta_i(d_j)= \delta_{i,j}$ for $1 \leq i,j \leq n$. For ${\bf{m}}= (m_1,m_2, \ldots, m_n)$, set $\delta_{{\bf{m}}}= \sum_{i=1}^n m_{i}\, \delta_{i}$. By abuse of notation, we shall write ${\bf m}$ in place of $\delta_{\bf m}$.\par
	 The extended Witt algebra is defined as $\bar{\EuScript{G}}:=\mathcal{W}_n \ltimes A_n$ with the following bracket relations along with \eqref{a45}
	\begin{equation*}
		[D({\bf{u}}, {\bf{r}} ), t^{\bf{m}}]= ({\bf{u}}|{\bf{m}})\, t^{\bf{r}+\bf{m}}, \quad [t^{\bf{r}}, t^{\bf{m}}]=0, \quad \text{for all}\,\, {\bf{u}} \in \C^n,\, {\bf{r}}, {\bf{m}} \in \Z^n.
	\end{equation*}
	This paper is concerned with the Lie algebra $\EuScript{G}_n:=\D \ltimes \mathcal{A}_n$ which is a subalgebra of $\mathcal{W}_n \ltimes \mathcal{A}_n$. For notational simplicity, we denote $\EuScript{G}_n$ by $\EuScript{G}$. We call this algebra by extended divergence zero Lie algebra. The following bracket formulas will be very useful for our study.
	\begin{equation*}
		[D_{ij}({\bf{r}}), D_{kl}({\bf{s}})]=r_js_k D_{il}({\bf{r}}+{\bf{s}})-r_js_l D_{ik}({\bf{r}}+{\bf{s}})-r_is_k D_{jl}({\bf{r}}+{\bf{s}})+r_is_l D_{jk}({\bf{r}}+{\bf{s}}),\vspace{0.1cm}
	\end{equation*}
	\begin{equation*}
		[D_{ij}({\bf{r}}), t^{\bf{m}}]=(r_jm_i-r_im_j)\, t^{\bf{r}+\bf{m}}, \quad [d_k, D_{ij}({\bf{r}})]=r_k\, D_{ij}({\bf{r}}),
		\vspace{0.15cm}
	\end{equation*}
	for all ${\bf{r}},\, {\bf{s}},\,{\bf{m}} \in \Z^n, \, 1 \leq i,j,k,l \leq n$. In particular,
	\begin{equation*}
		[D_{ij}({\bf{r}}), D_{ij}({\bf{s}})]=(r_js_i-r_is_j)\,D_{ij}({\bf{r}}+{\bf{s}}).
		\vspace{0.15cm}
	\end{equation*}
	Notice that $\La$ admits a natural $\Z^n$-grading and its homogeneous subspaces are given by
	\begin{equation}\label{a24}
		\La_{\bf m}=
		\begin{cases}
			\displaystyle \sum_{1\leq i<j\leq n}\C D_{ij}({\bf m})\oplus \C t^{\bf m}, & {\bf m}\neq {\bf 0},\\[0.5cm]
			\displaystyle \bigoplus_{j=1}^{n}\C d_j \oplus \C t^{\bf 0}, & {\bf m}={\bf 0}.
		\end{cases}
	\end{equation}
	\begin{remark}\label{a65}
		Note that $\La= \D \ltimes \mathcal{A}'_n \oplus \C t^{\bf 0}$ is a direct sum of ideals, where $\mathcal{A}'_n= \oplus_{{\bf{m}} \in \Z^n\setminus \{\bf{0}\}} \,\C t^{\bf{m}}$. Due to Schur's Lemma, the
		action of $t^{\bf 0}$ is a scalar and does not affect the module structure for an irreducible $\La$-module. Thus, we can change
		the value of $t^{\bf 0}$ appropriately due to our convenience.
	\end{remark}
	Now we recall some basic definitions.\par
	\noindent\textbf{Trivial module.} Let $L$ be a Lie algebra and let $V$ be an $L$-module. We say that $V$ is a \emph{trivial module} if
	$x\cdot v=0\,\, \text{for all } x\in L,\ v\in V$.
	A vector $v\in V$ is called a \emph{trivial vector} if it is annihilated by $L$, i.e.,
	$x\cdot v=0 \,\, \text{for all } x\in L$.
	
	\medskip
	\noindent\textbf{Harish-Chandra module.}
	A $\EuScript{G}$-module $V$ is called a \emph{Harish-Chandra module} if the following conditions hold:
	\begin{enumerate}
		\item $V$ is a weight module with respect to $\mathcal{H}=\text{Span}\,\{d_1,d_2,\ldots,d_n\}$, i.e.,
		\[
		V=\bigoplus_{\lambda\in \mathcal{H}^*} V_\lambda, \qquad \text{where} \,V_\lambda=\{v\in V : d_i\cdot v=\lambda(d_i)v,\ 1\leq i\leq n\}.
		\]
		\item  $\dim V_\lambda<\infty,\,\, \text{for all } \lambda\in \mathcal{H}^*$.
	\end{enumerate}
	The set
	$P(V)=\{\lambda\in \mathcal{H}^* : V_\lambda\neq 0\}$
	is called the \emph{set of weights} of $V$. Any nonzero vector $v\in V_\lambda$ is called a \emph{weight vector} of weight $\lambda$.
	
	\medskip
	\noindent\textbf{Cuspidal Module.} A Harish-Chandra module $V$ is said to be cuspidal if the dimensions of its weight spaces are uniformly bounded, i.e., there exists $M \in \N$ such that dim $V_{\lambda} < M$, $\forall {\lambda} \in \mathcal{H}^{*}$. These modules are also known as uniformly bounded modules.
	
	\medskip
	\noindent\textbf{$\Z^n$-Graded Module.} A $\La$-module $V$ is called $\Z^n$-graded if $V$ can be written as \\$V=\oplus_{{\bf{m}} \in \Z^n} V_{\bf m}$ such that $\La_{\bf p} V_{\bf m} \subseteq V_{\bf m + \bf p}$. A $\Z^n$-graded module is said to be irreducible if it has no nonzero proper graded submodule. 
	
	\medskip
	\noindent\textbf{Highest Weight Module.}
	Suppose that a Lie algebra $L$ admits a triangular decomposition
	$L=L^{-}\oplus L^0\oplus L^{+}$.
	An $L$-module $V$ is called a \emph{highest weight module} if there exists a nonzero vector $v\in V$ such that
	\[
	L^{+}\cdot v=0
	\qquad \text{and} \qquad
	\mathcal{U}(L) v=V.
	\]
	Such a vector $v$ is called a \emph{highest weight vector}.
	
	\medskip
	\noindent\textbf{Change of Coordinates.} Let $A$ be an $n \times n$ matrix in $GL_n{(\Z)}$ with determinant $\pm 1$. Then the assignment
	\begin{center}
		$D({\bf{u}}, {\bf{r}}) \mapsto D({\bf{Bu}},{\bf{Ar}}), \quad \quad \quad t^{{\bf{r}}} \mapsto t^{\bf{Ar}}$ 
	\end{center}
	defines a Lie algebra automorphism $T_A$ of $\EuScript{G}$, where $B=(A^t)^{-1}$ and $A^{t}$ denotes the transpose of $A$.  \par
	Alternatively, we can define another Lie algebra automorphism of $\EuScript{G}$ in a different way. Let $\{{\bm{\al}}_1,\bm{\al}_2,\ldots,\bm{\al}_n\}$ be a $\Z$-basis of $\mathbb{Z}^n$ and $\{\bm{\beta}_1,\bm{\beta}_2,\ldots, \bm{\beta}_n\}$ be the corresponding dual basis in $\mathbb{R}^n$. Define $\Phi: \EuScript{G} \rightarrow \EuScript{G}$ by
	\vspace{0.15cm}
	\begin{center}
		$D({\bf{u}}, {\bf{r}}) \mapsto D(\sum_{i=1}^{n}u_i\bm{\beta}_i\,,\,\,\, \sum_{i=1}^{n}r_i\bm{\al}_i), \quad \quad t^{\bf{m}} \mapsto t^{\sum_{i=1}^{n}m_i\bm{\al}_i}.\hspace{0.4cm} \hspace{0.4cm}  $
	\end{center}
	It is straightforward to verify that both $T_A$ and $\Phi$ are Lie algebra automorphisms of $\EuScript{G}$. These automorphisms yield a new extended divergence zero Lie algebra from the original one. We shall refer to this procedure as a \emph{change of coordinates}.
	
	\medskip
	\noindent\textbf{Jet Modules.} A $\EuScript{G}$-module $V$ is called a jet module if it satisfies the following properties:
	\begin{enumerate}
		\item $V$ is a weight module with respect to the Cartan subalgebra $\mathcal{H}$.
		\item As an $\mathcal{A}_n$-module $V$ is free with finite rank.
		\item The action of $\D$ on $V$ is compatible with the $\mathcal{A}_n$-module structure in the sense that
		\begin{equation*}
			X(fv)=(Xf)v+f(Xv),
			\quad \forall\, X\in \D,\ f\in \mathcal{A}_n,\ v\in V.
		\end{equation*}
	\end{enumerate}
	\begin{remark}
		For an irreducible module $V$ over $\La$, there exists $\lambda \in \mathcal{H}^{*}$ such that $P(V) \subseteq \{\lambda+{\bf{m}}: {\bf{m}} \in \Z^n\}$.
	\end{remark}
	\subsection{Representation of extended Witt algebras}Let $\mathfrak{gl}_n$ be the Lie algebra of all $n \times n$ complex matrices and ${\mathfrak{sl}_{n}}$ denote the subalgebra of $\mathfrak{gl}_n$ consisting of trace zero matrices. Let $E_{ij}$ denote the $n \times n$ matrix with 1 in the $(i,j)$-th position and 0 elsewhere. Then $H= \text{Span}\,\{h_i:=E_{i,i}-E_{i+1,i+1}: 1 \leq i \leq n-1\}$ is a Cartan subalgebra of ${\mathfrak{sl}_{n}}$. The corresponding fundamental weights $\omega_1, \ldots, \omega_{n-1}$ are uniquely 
	defined by $\omega_i(h_j)=\delta_{ij}$, for all $1\leq i,j\leq n-1$.\par
	For a dominant integral weight $\lambda$ of $\mathfrak{sl}_n$, let $V(\lambda)$ denote the finite dimensional irreducible highest weight module for ${\mathfrak{sl}_{n}}$. We extend $V(\lambda)$ to a $\mathfrak{gl}_n$-module by assuming the identity matrix to act as the scalar $c\in \C$, and denote the resulting module by $V(\lambda,c)$. Given any ${\boldsymbol{\alpha}} \in \C^n$, it is well known that we can define a $\mathcal{W}_n$-module structure on $F^{\alpha}(\lambda, c):=V(\lambda,c) \otimes \mathcal{A}_n$ by defining 
	\begin{equation}\label{a20}
		D({\bf{u}}, {\bf{r}})\,(v \otimes t^{\bf{s}})=\,({\bf{u}}|{\bf{s}}+{\boldsymbol{\alpha}})\, v \otimes t^{{\bf{r}}+{\bf{s}}} + (({\bf{r}}{\bf{u}}^T).v) \otimes t^{{\bf{r}}+{\bf{s}}},
	\end{equation}
	for all ${\bf{u}} \in \C^n, {\bf{r}},\,{\bf{s}} \in \Z^n$ and $v \in V(\lambda,c)$.
	\begin{theorem}[\cite{guo}, Theorem 2.1]
		$F^{\alpha}(\lambda, c)$ is irreducible if and only if either $(\lambda,c, {\boldsymbol{\alpha}}) \notin \{0\} \times \{0,n\} \times \Z^n$ or $(\lambda,c) \neq (\omega_k, k)$, for all $1 \leq k \leq n$.
	\end{theorem}
	Furthermore, for any $e\in \C$, the module $F^{\alpha}(\lambda,c)$ can be extended to a $(\mathcal{W}_n\ltimes \mathcal{A}_n)$-module by defining
	\begin{equation*}
		t^{\bf r}(v\otimes t^{\bf s})=ev\otimes t^{{\bf r}+{\bf s}},\quad \forall {\bf r},{\bf s}\in \Z^n, v\in V(\lambda,c).
	\end{equation*}
	We denote the resulting $(\mathcal{W}_n\ltimes \mathcal{A}_n)$-module by
	$F^{\alpha}(\lambda,c,e)$.
	
	\begin{theorem}[\cite{guo}, Theorem 2.1]\label{a21}
		As a ($\mathcal{W}_n \ltimes \mathcal{A}_n)$-module $F^{\alpha}(\lambda, c,e)$ is irreducible if $e \neq 0$.
	\end{theorem}
	We now define a class of $(\D \ltimes \mathcal{A}_n)$-modules by restricting the $(\mathcal{W}_n \ltimes \mathcal{A}_n)$-module action given in \eqref{a20} to the subalgebra $(\D \ltimes \mathcal{A}_n)$. Observe that $({\bf{u}}| {\bf{r}})=0$ implies that the trace of ${\bf{r}}{\bf{u}}^T$ is zero. Hence ${\bf r}{\bf u}^{T}\in \mathfrak{sl}_n$, and the resulting $(\D \ltimes \mathcal{A}_n)$-modules are therefore of the form $V(\lambda)\otimes \mathcal{A}_n$. More explicitly, the action is given by
	\begin{equation*}
		\begin{aligned}
			D_{ij}({\bf{r}})(v \otimes t^{\bf{s}})=& (r_j(s_i+\alpha_i)-r_i(s_j+\alpha_j))\,v \otimes t^{{\bf{r}}+{\bf{s}}}\\
			&\,\, + \bigg(r_i \sum_{\substack{k=1 \\ k \neq i}}^{n}r_k E_{ki} - r_j \sum_{\substack{k=1 \\ k \neq j}}^{n}r_k E_{kj} +r_ir_j(E_{ii}-E_{jj})\bigg).v \otimes t^{{\bf{r}}+{\bf{s}}}\\
			d_i(v \otimes t^{\bf{s}})=&(s_i+\alpha_i)\,v \otimes t^{\bf{s}}, \quad t^{\bf{r}}(v \otimes t^{\bf{s}})=ev \otimes t^{{\bf{r}}+{\bf{s}}}.
		\end{aligned}
	\end{equation*}
	We denote this module by $V(\lambda, \boldsymbol{\alpha},e)$. By Theorem \ref{a21}, it follows that $V(\lambda, \boldsymbol{\alpha},e)$ is irreducible whenever $e \neq 0$.
	
	The following elementary lemma will be used repeatedly throughout the paper.
	
	\begin{lemma}[\cite{sou}, Lemma 3.2]
		Every nonzero module over a Lie algebra has a nonzero irreducible subquotient.
	\end{lemma}
	  
	\section{Classification of Irreducible cuspidal modules over $\EuScript{G}$}\label{a94}
	In this section, we classify the irreducible cuspidal $\EuScript{G}$-modules where $\mathcal{A}_n'$ acts non-trivially. Throughout this section, $V$ will always denote an irreducible cuspidal $\EuScript{G}$-module with non-trivial $\mathcal{A}_n'$-action. We begin with two lemmas from \cite{guo}, which will play an important role in the classification.

\begin{lemma}[\cite{guo}, Lemma 3.1]\label{a1}
	Let $V=\oplus_{{\bf{m}} \in \Z^n} V_{\bf{m}}$ be an irreducible uniformly bounded $\Z^n$-graded module over $\mathcal{A}_n$. Then dim $V_{\bf{m}} \leq 1$ for all ${\bf{m}} \in \Z^n$.
\end{lemma}

\begin{lemma}[\cite{guo}, Lemma 3.2]\label{a2}
	Let $V$ be an irreducible $\EuScript{G}$-module. If there exist $f \in \mathcal{U}(\mathcal{A}_n)$ and $v \in V$ such that $fv=0$, then $f$ is locally nilpotent on $V$.
\end{lemma}
\begin{lemma}\label{a3}
	For ${\bf r},{\bf s}\in \Z^n$ with ${\bf s}\notin \Q {\bf r}$, there exists ${\bf u}\in \C^n$ such that
	\begin{center}
		$({\bf{u}}\,|\,{\bf{s}}) \neq 0$ \quad and \quad $({\bf{u}}\,|\,{\bf{r}}+{\bf{s}}) = 0$.
	\end{center}
	\begin{proof}
		Since ${\bf s}\notin \Q {\bf r}$, we may choose $1\leq i,j\leq n$ such that
		$s_jr_i-s_ir_j\neq 0$. Now take ${\bf{u}}=(r_i+s_i)\,{\bf{e}}_j-(r_j+s_j)\,{\bf{e}}_i$, then clearly $({\bf{u}}\,|\,{\bf{r}}+{\bf{s}}) = 0$ and $({\bf{u}}|{\bf{s}})=s_jr_i-s_ir_j \neq 0$.
	\end{proof}
\end{lemma}

\begin{prop}\label{a6}
	Let $V$ be an irreducible cuspidal $\EuScript{G}$-module. Then either $t^{\bf r}$ acts injectively on $V$ for every ${\bf r}  \neq {\bf 0}$, or $t^{\bf r}$ acts locally nilpotently on $V$ for every ${\bf r} \neq {\bf 0}$.
	\begin{proof}
		Suppose first that $t^{\bf r}$ is not injective for some ${\bf r}\neq {\bf 0}$. Then there exists $v\in V$ such that $t^{\bf r}v=0$, and Lemma \ref{a2} implies that $t^{\bf r}$ acts locally nilpotently on $V$. 	Now assume that $t^{\bf r}$ is injective for some ${\bf r}\neq {\bf 0}$. We show that this forces $t^{\bf s}$ to be injective for every ${\bf s}\neq {\bf 0}$.
		
		We first deal with the case where ${\bf s}\notin \Q{\bf r}$ and $t^{\bf{s}}$ is not injective. Then by Lemma \ref{a2}, $t^{\bf s}$ acts locally nilpotently on $V$. Consequently, $t^{-{\bf s}}t^{\bf s}$ also acts locally nilpotently on $V$, and hence on each homogeneous subspace of $V$. Since $V$ is cuspidal, there exists $N\in \N$ such that
		$(t^{-{\bf s}}t^{\bf s})^N V=\{0\}$. By Lemma \ref{a3}, we may choose ${\bf u}\in \C^n$ such that $({\bf{u}}|{\bf{s}}) \neq 0$ and $({\bf{u}}|{\bf{r}}+{\bf{s}}) = 0$. Then for every $v\in V$,
		\begin{equation*}
			0=D({\bf{u}}, {\bf{r}}+{\bf{s}})\,(t^{\bf{-s}})^N\,(t^{\bf{s}})^N\, v=N({\bf{u}}|{\bf{s}}) \bigg(t^{{\bf{r}}+2{\bf{s}}}\, (t^{\bf{-s}})^N \,(t^{\bf{s}})^{N-1}- t^{\bf{r}}\,(t^{\bf{-s}})^{N-1}\, (t^{\bf{s}})^N \bigg)v.
		\end{equation*}
		Applying $t^{\bf s}$ to both sides yields
		\begin{equation*}
			t^{\bf{r}}\,(t^{\bf{-s}})^{N-1}\, (t^{\bf{s}})^{N+1}\, V=0.
		\end{equation*}
		Since $t^{\bf r}$ is injective, we obtain
		\begin{center}
			$(t^{-{\bf s}})^{N-1}(t^{\bf s})^{N+1}V=0$.
		\end{center}
		Applying $t^{\bf s}$ successively another $(N-2)$ more times and proceeding as above, we eventually obtain
		$(t^{\bf s})^{2N}V=\{0\}$. Next choose ${\bf w}\in \C^n$ such that $({\bf{w}}|{\bf{s}}) \neq 0$ and $({\bf{w}}|{\bf{r}}-{\bf{s}}) = 0$. Then
		\begin{equation*}
			0= D({\bf{w}}, {\bf{r}}-{\bf{s}})\,(t^{\bf{s}})^{2N}\, v=2N({\bf{w}}|{\bf{s}})\,t^{\bf{r}}\,(t^{\bf{s}})^{2N-1}\,v, \quad \forall v \in V.
		\end{equation*}
		Again using the injectivity of $t^{\bf r}$, we deduce
		$(t^{\bf s})^{2N-1}V=0$. Continuing this process, we eventually obtain $t^{\bf{r}}V=\{0\}$, contradicting the injectivity of $t^{\bf r}$. Therefore $t^{\bf s}$ is injective for all ${\bf s}\notin \Q {\bf r}$.
		
		It remains to consider ${\bf s}\in \Q{\bf r}\,\cap\, \Z^n \setminus \{\bf 0\}$. Suppose that $t^{\frac{p}{q}\bf{r}}$ is not injective for some $\frac{p}{q} \in \Q$. Then Lemma \ref{a2} implies that $t^{\frac{p}{q}\bf{r}}$ acts locally nilpotently, and hence $(t^{\bf{r}})^{-p}\,(t^{\frac{p}{q}\bf{r}})^q$
		acts locally nilpotently on $V$. Hence there exists $N' \in \N$ such that $\big((t^{\bf{r}})^{-p}\,(t^{\frac{p}{q}\bf{r}})^q \big)^{N'} \,V=\{0\}$.
		Since $t^{\bf r}$ is injective, we obtain
		$\Bigl(t^{\frac{p}{q}{\bf r}}\Bigr)^{qN'}V=0$. Let $N_0$ be the smallest natural number such that $\Bigl(t^{\frac{p}{q}{\bf r}}\Bigr)^{N_0}V=0$.
		Now fix ${\bf m}\in \Z^n$ with ${\bf m}\notin \Q {\bf r}$. By Lemma \ref{a3}, we can choose ${\bf u}'\in \C^n$ such that
		$({\bf u}'|\, \frac{p}{q}{\bf r})\neq 0\,\, \text{and} \,\,
		({\bf u}'\mid {\bf m}-\frac{p}{q}{\bf r})=0$.
		Then
		\begin{equation*}
			0=D({\bf{u}}', {\bf{m}}-\frac{p}{q}{\bf{r}})\, \Bigl(t^{\frac{p}{q}\bf{r}}\Bigr)^{N_0}\,V= N_0\, ({\bf{u}}'|\,\frac{p}{q}{\bf{r}}) \,t^{\bf{m}} \,\Bigl(t^{\frac{p}{q}\bf{r}}\Bigr)^{N_0 -1}\,V.
		\end{equation*}
		By the minimality of $N_0$, we have $(t^{\frac{p}{q}\bf{r}})^{N_0 -1}\,V \neq 0$. Hence $t^{\bf m}$ is not injective, contradicting the previous case. Hence $t^{\bf s}$ is injective for all ${\bf s}\in \Q{\bf r}\,\cap\, \Z^n\setminus \{{\bf 0}\}$ and this completes the proof.
	\end{proof}
\end{prop}
\begin{prop}\label{a7}
	Let $V$ be an irreducible cuspidal $\EuScript{G}$-module. Suppose that $t^{\bf r}$ acts locally nilpotently on $V$ for every ${\bf r} \neq {\bf 0}$, then
	$\mathcal{A}_n'V=0$.
	\begin{proof}
		Since each $t^{\bf r}$ (${\bf r}\neq {\bf 0}$) acts locally nilpotently on $V$, and $V$ is cuspidal, there exists $N\in \N$ such that
		$(t^{-{\bf r}}t^{\bf r})^N V=\{0\}$ for every ${\bf{r}} \neq {\bf{0}}$. Now fix ${\bf{r}} \neq {\bf{0}}$. Let ${\bf s}_1 \notin \mathbb{Q}{\bf{r}}$ and choose ${\bf{u}}_1 \in \C^n$ such that $({\bf{u}}_1|{\bf{r}}) \neq 0$ and $({\bf{u}}_1|{\bf{r}}+{\bf{s}}_1) = 0$. Then for any $v \in V$, we have
		\begin{equation*}
			0=D({\bf{u}}_1, {\bf{r}}+{\bf{s}}_1)\,(t^{\bf{-r}})^N\,(t^{\bf{r}})^N\, v=N({\bf{u}}_1|{\bf{r}}) \bigg( t^{{\bf{s}}_1 +2{\bf{r}}}\, (t^{\bf{-r}})^N \,(t^{\bf{r}})^{N-1}- t^{\bf{s}_1}\,(t^{\bf{-r}})^{N-1}\, (t^{\bf{r}})^N     \bigg)v.
		\end{equation*}
		Applying $t^{\bf r}$ to both sides, we obtain
		\begin{equation*}
			t^{{\bf s}_1}(t^{-{\bf r}})^{N-1}(t^{\bf r})^{N+1}\,V=\{0\},
			\qquad
			\forall\, {\bf s}_1\notin \Q {\bf r}.
		\end{equation*}
		\noindent\textbf{Claim 1:} For $1 \leq j \leq N$ and ${\bf{s}}_1,\ldots, {\bf{s}}_j \in \Z^n \setminus \mathbb{Q}{\bf{r}}$, we have 
		\begin{equation*}
			t^{{\bf{s}}_1}\cdots t^{{\bf{s}}_j}\,(t^{\bf{-r}})^{N-j}\, (t^{\bf{r}})^{N+j}\, V=0.
		\end{equation*}
		We have already proved the base case $j=1$. Assume the statement holds for some $1<k<N$. Let ${\bf{s}}_{k+1}  \notin \mathbb{Q}{\bf{r}}$ and choose ${\bf{u}}_{k+1} \in \C^n$ such that $({\bf{u}}_{k+1}|{\bf{r}}) \neq 0$ and $({\bf{u}}_{k+1}\,|\,{\bf{r}}+{\bf{s}}_{k+1}) = 0$. For $v \in V$, we have
		\begin{equation*}
			\begin{aligned}
				0&=D({\bf{u}}_{k+1}\, , \,{\bf{r}}+{\bf{s}}_{k+1})\, t^{\bf{s}_1}\cdots t^{\bf{s}_k}\,(t^{\bf{-r}})^{N-k}\, (t^{\bf{r}})^{N+k}\, v\\
				&=\sum_{i=1}^{k} ({\bf{u}}_{k+1}\,|\,{\bf{s}}_{i}) t^{\bf{s}_1}\cdots t^{\bf{s}_{i-1}}\, t^{{\bf{r}}+{\bf{s}}_{i}+{\bf{s}}_{k+1}}\, t^{\bf{s}_{i+1}}\, t^{\bf{s}_k}\,(t^{\bf{-r}})^{N-k}\, (t^{\bf{r}})^{N+k}\, v\\
				&\quad \quad -(N-k) ({\bf{u}}_{k+1}\,|\,{\bf{r}})\, t^{\bf{s}_1}\cdots t^{\bf{s}_{k+1}}\,(t^{\bf{-r}})^{N-k-1}\, (t^{\bf{r}})^{N+k}\,v\\
				&\quad \quad + (N+k) ({\bf{u}}_{k+1}\,|\,{\bf{r}})\, t^{2{\bf{r}}+{\bf{s}}_{k+1}}\, t^{\bf{s}_1}\cdots t^{\bf{s}_k}\,(t^{\bf{-r}})^{N-k}\, (t^{\bf{r}})^{N+k-1}\, v.
			\end{aligned}
		\end{equation*}
		Since ${\bf r}+{\bf s}_i+{\bf s}_{k+1}\notin \Q {\bf r},\,\, \forall 1 \leq i \leq k$, each term in the above summation vanishes by the induction hypothesis. Now applying $t^{\bf{r}}$ to both sides and using induction hypothesis, we get
		\begin{equation*}
			t^{\bf{s}_1}\cdots t^{\bf{s}_{k+1}}\,(t^{\bf{-r}})^{N-k-1}\, (t^{\bf{r}})^{N+k+1}\,V=0.
		\end{equation*} 
		This proves the claim. Consequently, we have 
		\begin{equation}\label{a50}
			t^{\bf{s}_1}\cdots t^{\bf{s}_N}\,(t^{\bf{r}})^{2N}\, V=0,\quad  \text{for all}\,\, {\bf{s}}_1,\ldots, {\bf{s}}_N \in \Z^n \setminus \mathbb{Q}{\bf{r}}.
		\end{equation}
		Again take ${\bf{s}}_{N+1}  \notin \mathbb{Q}{\bf{r}}$ and choose ${\bf{u}}_{N+1} \in \C^n$ such that $({\bf{u}}_{N+1}|{\bf{r}}) \neq 0$ and $({\bf{u}}_{N+1}\,|\,{\bf{s}}_{N+1}-{\bf{r}}) = 0$. We have,
		\begin{equation*}
			\begin{aligned}
				0&=D({\bf{u}}_{N+1}\, , \,{\bf{s}}_{N+1}-{\bf{r}})\, t^{\bf{s}_1}\cdots t^{\bf{s}_N}\, (t^{\bf{r}})^{2N}\, V\\
				&=\sum_{i=1}^{N} ({\bf{u}}_{N+1}\,|\,{\bf{s}}_{i}) t^{\bf{s}_1}\cdots t^{\bf{s}_{i-1}}\, t^{{\bf{s}}_{N+1}-{\bf{r}}+{\bf{s}}_{i}}\, t^{\bf{s}_{i+1}}\, (t^{\bf{r}})^{2N}\, V\\
				&\quad \quad + 2N ({\bf{u}}_{N+1}\,|\,{\bf{r}})\, t^{\bf{s}_1}\cdots t^{\bf{s}_{N+1}}\, (t^{\bf{r}})^{2N-1}\, V
			\end{aligned}
		\end{equation*}
		Then \eqref{a50} implies that
		\begin{equation}\label{a4}
			t^{\bf{s}_1}\cdots t^{\bf{s}_{N+1}}\,(t^{\bf{r}})^{2N-1}\, V=0,\quad  \text{for all}\,\, {\bf{s}}_1,\ldots, {\bf{s}}_{N+1} \in \Z^n \setminus \mathbb{Q}{\bf{r}}.
		\end{equation}
		Continuing this process another $(2N-1)$ times, we finally get
		\begin{equation}\label{a5}
			\bigg(\prod_{j=1}^{3N}t^{{\bf{s}}_j}\bigg)V=0, \quad  \text{for all}\,\, {\bf{s}}_1,\ldots, {\bf{s}}_{3N} \in \Z^n \setminus \mathbb{Q}{\bf{r}}.
		\end{equation}
		\noindent\textbf{Claim 2:} $(\mathcal{A}_n')^{3N}\,V=\{0\}$.\vspace{0.15cm}\\
		Given arbitrary ${\bf{s}}_1,\ldots, {\bf{s}}_{3N} \in \Z^n \setminus \{{\bf{0}}\}$, we may choose ${\bf{r}} \in \Z^n \setminus \{{\bf{0}}\}$ such that ${\bf{s}}_1,\ldots, {\bf{s}}_{3N} \in \Z^n \setminus \mathbb{Q}{\bf{r}}$. Then \eqref{a5} yields $\big(\prod_{j=1}^{3N}t^{{\bf{s}}_j}\big)V=0$, hence the claim follows.
		
		Let $N_0$ be the smallest positive integer such that
		$(\mathcal{A}_n')^{N_0}V=0$.
		Then there exist $v\in V$ and ${\bf s}_1,\dots,{\bf s}_{N_0-1}\in \Z^n\setminus \{{\bf 0}\}$ such that
		$w:=t^{{\bf s}_1}\cdots t^{{\bf s}_{N_0-1}}v\neq 0$. Therefore
		$t^{\bf s}w=0,
		\,\, \text{for all } {\bf s}\neq {\bf 0}$. Then $V'=\{v \in V: t^{\bf{s}}v=0,\, \forall {\bf{s}} \neq {\bf{0}} \}$ is a non-zero $\EuScript{G}$-submodule of $V$. Since $V$ is irreducible, we must have $V'=V$, therefore,
		$\mathcal{A}_n' \cdot V=0$.
	\end{proof}
\end{prop}
\begin{theorem}\label{a19}
	Suppose that $\mathcal{A}_n'$ acts nontrivially on $V$. Then there exists a nonzero scalar $c$ such that
	\begin{equation*}
		t^{\bf r}\,t^{\bf s}.v=c \, t^{{\bf r}+{\bf s}}.v
		\qquad \text{for all }\,\, {\bf r},{\bf s}\in \Z^n \text{ and } v\in V.
	\end{equation*}
	\begin{proof}
		By remark \ref{a65}, we may assume that $t^{\bf{0}}$ acts on $V$ by $c \neq 0$. Write
		\begin{equation*}
			V=\bigoplus_{{\bf r}\in \Z^n} V_{\bf r}, \quad \text{where} \,\,\,\,V_{\bf r}=\{v\in V : d_i v=(\mu_i+r_i)v,\ 1\leq i\leq n\}.
		\end{equation*}
		Since $V$ is a uniformly bounded $\Z^n$-graded $\mathcal{A}_n$-module, Lemma \ref{a1} implies that $V$ contains an irreducible $\mathcal{A}_n$-submodule
		\begin{equation*}
			V'=\bigoplus_{{\bf r}\in \Z^n} V'_{\bf r}
			\quad \text{with} \quad V'_{\bf r}\subseteq V_{\bf r}\,\,\, \text{and} \,\,\, \dim \,V'_{\bf r}\leq 1
			\quad \text{for all}\,\, {\bf r}\in \Z^n.
		\end{equation*}
		Since $\mathcal{A}_n'$ acts nontrivially on $V$, Propositions \ref{a6} and \ref{a7} show that each $t^{\bf r}$ (${\bf r}\neq {\bf 0}$) acts injectively on $V$. Hence $\dim\, V'_{\bf r}=1$ for all ${\bf r}\in \Z^n$. Fix $0\neq v_0\in V'_{\bf 0}$. Then for every ${\bf r},{\bf s}\in \Z^n$, there exists a scalar
		$c_{{\bf r},{\bf s}}\in \C^*$
		such that
		\begin{equation}\label{a8}
			t^{\bf r}t^{\bf s}v_0=c_{{\bf r},{\bf s}}\, t^{{\bf r}+{\bf s}}v_0.
		\end{equation}
		Note that we can adjust the value of $c$ so that $c_{{\bf{e}_1}, {\bf{e}_1}}=c_{{\bf{e}_1}, {\bf{-e}}_1}$. By the definition, we have $c_{{\bf{r}}, {\bf{s}}}=c_{{\bf{s}}, {\bf{r}}}$ and $c_{{\bf{r}}, {\bf{0}}}=c$.
		Set 
		\begin{center}
			$f_{{\bf{r}}, {\bf{s}}}= t^{\bf{r}}\,t^{\bf{s}}-c_{{\bf{r}}, {\bf{s}}}\, t^{{\bf{r}}+{\bf{s}}}, \quad \quad \forall\,{\bf{r}}, {\bf{s}} \in \Z^n$.
		\end{center}
		Then $f_{{\bf r},{\bf s}}v_0=0$, so by Lemma \ref{a2}, $f_{{\bf r},{\bf s}}$ acts locally nilpotently on $V$ for all ${\bf r},{\bf s}\in \Z^n$. Let $\dim V_{\bf 0}=L$, and choose a basis $\{v_1,\dots,v_L\}$ of $V_{\bf 0}$. Since each $t^{\bf r}$ (${\bf r}\neq {\bf 0}$) acts injectively, all homogeneous components $V_{\bf r}$ have dimension $L$, and $\{t^{\bf r}v_1,\dots,t^{\bf r}v_L\}$ is a basis of $V_{\bf r}$ for every ${\bf r}\in \Z^n$. Thus, for each ${\bf r},{\bf s}\in \Z^n$, there exists a matrix $C_{{\bf r},{\bf s}}\in M_L(\C)$ such that
		\begin{equation*}
			(t^{\bf{r}}\, t^{\bf{s}}\,v_1, t^{\bf{r}}\, t^{\bf{s}}\,v_2, \ldots, t^{\bf{r}}\,t^{\bf{s}}\,v_L)=(t^{{\bf{r}}+{\bf{s}}}\,v_1,t^{{\bf{r}}+{\bf{s}}}\,v_2,\ldots, t^{{\bf{r}}+{\bf{s}}}\,v_L)\,C_{{\bf{r}}, {\bf{s}}}
		\end{equation*}
		Since $\mathcal{A}_n$ is commutative, the matrices $C_{{\bf r},{\bf s}}$ commute pairwise, i.e.,
		\begin{center}
			$C_{{\bf{r}_1}, {\bf{s}_1}}\,C_{{\bf{r}_2}, {\bf{s}_2}}=C_{{\bf{r}_2}, {\bf{s}_2}}\,C_{{\bf{r}_1}, {\bf{s}_1}}, \quad \forall \,\,{\bf{r}_1},\,{\bf{r}_2},\,{\bf{s}_1},\,{\bf{s}_2} \in \Z^n$.
		\end{center}
		Hence, by Lie's theorem, there exists an invertible matrix $Q$ such that
		$D_{{\bf r},{\bf s}}:=Q^{-1}C_{{\bf r},{\bf s}}Q$
		is upper triangular for all ${\bf r},{\bf s}\in \Z^n$. Setting 
		\begin{equation*}
			(u_1,u_2, \ldots, u_L)=(v_1,v_2, \ldots, v_L)\,Q,
		\end{equation*}
		we obtain
		\begin{equation*}
			(t^{\bf{r}}\, t^{\bf{s}}\,u_1, t^{\bf{r}}\, t^{\bf{s}}\,u_2, \ldots, t^{\bf{r}}\,t^{\bf{s}}\,u_L)=(t^{{\bf{r}}+{\bf{s}}}\,u_1,t^{{\bf{r}}+{\bf{s}}}\,u_2,\ldots, t^{{\bf{r}}+{\bf{s}}}\,u_L)\,D_{{\bf{r}}, {\bf{s}}}.
		\end{equation*}
		Consequently,
		\begin{equation*}
			(f_{{\bf{r}}, {\bf{s}}}\,u_1, f_{{\bf{r}}, {\bf{s}}}\,u_2, \ldots, f_{{\bf{r}}, {\bf{s}}}\,u_L)=(t^{{\bf{r}}+{\bf{s}}}\,u_1,t^{{\bf{r}}+{\bf{s}}}\,u_2,\ldots, t^{{\bf{r}}+{\bf{s}}}\,u_L)\,(D_{{\bf{r}}, {\bf{s}}}- c_{{\bf{r}}, {\bf{s}}}\, I_L),
		\end{equation*}
		where $I_L$ is the $L\times L$ identity matrix. Since $f_{{\bf r},{\bf s}}$ is locally nilpotent, all the diagonal entries of $D_{{\bf{r}}, {\bf{s}}}- c_{{\bf{r}}, {\bf{s}}}\,I_L$ are zero. Therefore, 
		\begin{equation}\label{a9}
			f_{{\bf{r}}, {\bf{s}}}\,u_1=0,\quad (f_{{\bf{r}}, {\bf{s}}})^L\,u_L=0, \quad f_{{\bf{r}}, {\bf{s}}}\,u_j \in \sum_{i=1}^{j-1} \C\, t^{{\bf{r}}+{\bf{s}}}\,u_i, \quad \forall \,{\bf{r}}, {\bf{s}} \in \Z^n, \,\,2 \leq j \leq L.
		\end{equation}
		Now let ${\bf{r}}, {\bf{s}} \in \Z^n$ and take $w \in V_{-({\bf{r}}+{\bf{s}})}= \sum_{i=1}^{L} \C \,t^{-({\bf{r}}+{\bf{s}})}\, u_i$. Using \eqref{a9} we have
		\begin{equation*}
			f_{{\bf{r}}, {\bf{s}}}\,w \in \sum_{i=1}^{L-1} \C \, t^{-({\bf{r}}+{\bf{s}})}\, t^{{\bf{r}}+{\bf{s}}}u_i \subseteq \sum_{i=1}^{L-1} \C (f_{({\bf{r}}+{\bf{s}}), -({\bf{r}}+{\bf{s}})}+c_{({\bf{r}}+{\bf{s}}), -({\bf{r}}+{\bf{s}})}\,t^{\bf{0}}) u_i \subseteq \sum_{i=1}^{L-1} \C u_i.
		\end{equation*}	
		Set $\overline{V}=\sum_{i=1}^{L-1} \C u_i$. Then we have $f_{{\bf{r}}, {\bf{s}}}\,V_{-({\bf{r}}+{\bf{s}})} \in \overline{V}$. Moreover,
		\begin{equation}\label{a10}
			t^{-\bf{r}}\,t^{\bf{r}}\,u_L=(f_{{-\bf{r}}, {\bf{r}}}+c_{{-\bf{r}}, {\bf{r}}} t^{\bf{0}})\,u_L \equiv c c_{{-\bf{r}}, {\bf{r}}}\,u_L \quad (\text{mod}\,\,\overline{V}).
		\end{equation}
		Now let ${\bf{r}}, {\bf{s}}, {\bf{q}} \in \Z^n$, set ${\bf{a}}= {\bf{r}}+{\bf{s}}+{\bf{q}}$. Choose ${\bf{z}} \in \C^n$ such that $({\bf{z}}|{\bf{q}})=0$. Then
		\begin{equation*}
			\begin{aligned}
				0=\,&(t^{-\bf{a}})^L\,(D({\bf{z}},{\bf{q}}))^L\,(f_{{\bf{r}}, {\bf{s}}})^L\, u_L\\
				\equiv\, & (t^{-\bf{a}})^L\,([D({\bf{z}},{\bf{q}}), f_{{\bf{r}}, {\bf{s}}}])^L\,u_L \quad \quad (\text{mod}\,\overline{V}\,)\\
				\equiv \, & (t^{-\bf{a}})^L\,\bigg( ({\bf{z}}|{\bf{r}})\, t^{{\bf{q}}+{\bf{r}}}\,t^{\bf{s}} +({\bf{z}}|{\bf{s}})\, t^{\bf{r}}\,t^{{\bf{q}}+{\bf{s}}} - c_{{\bf{r}}, {\bf{s}}}\,({\bf{z}}|{\bf{r}}+{\bf{s}})\,t^{{\bf{q}}+{\bf{r}}+{\bf{s}}}         \bigg)^L\,u_L \quad \quad (\text{mod}\,\overline{V}\,)\\
				\equiv \, & (t^{-\bf{a}})^L\,\bigg(\hspace{-0.15cm} ({\bf{z}}|{\bf{r}})(f_{{\bf{q}}+{\bf{r}}, {\bf{s}}}+c_{{\bf{q}}+{\bf{r}}, {\bf{s}}}\,t^{{\bf{a}}})\hspace{-0.1cm}+\hspace{-0.1cm} ({\bf{z}}|{\bf{s}})(f_{{\bf{r}}, {\bf{q}}+ {\bf{s}}}+c_{{\bf{r}}, {\bf{q}}+ {\bf{s}}}\,t^{{\bf{a}}}) \hspace{-0.05cm}- \hspace{-0.05cm} c_{{\bf{r}}, {\bf{s}}}\,({\bf{z}}|{\bf{r}}+{\bf{s}})\,t^{{\bf{a}}} \hspace{-0.1cm}                                     \bigg)^L \hspace{-0.15cm}u_L \,\,\, (\text{mod}\,\overline{V})\\
				\equiv \, & (t^{-\bf{a}})^L\,\bigg( ({\bf{z}}|{\bf{r}})\,c_{{\bf{q}}+{\bf{r}}, {\bf{s}}}+ ({\bf{z}}|{\bf{s}})\,c_{{\bf{r}}, {\bf{q}}+ {\bf{s}}} -  c_{{\bf{r}}, {\bf{s}}}\,({\bf{z}}|{\bf{r}}+{\bf{s}})                                \bigg)^L(t^{\bf{a}})^L\,u_L \quad \quad (\text{mod}\,\overline{V})\\
			\end{aligned}
		\end{equation*}
		Using \eqref{a10}, we deduce that for every ${\bf{r}}, {\bf{s}}, {\bf{q}} \in \Z^n$,
		\begin{equation}\label{a11}
			({\bf{z}}|{\bf{r}})\,c_{{\bf{q}}+{\bf{r}}, {\bf{s}}}+ ({\bf{z}}|{\bf{s}})\,c_{{\bf{r}}, {\bf{q}}+ {\bf{s}}} =  ({\bf{z}}|{\bf{r}}+{\bf{s}})\,c_{{\bf{r}}, {\bf{s}}} \quad \quad \forall \,{\bf{z}} \in \C^n \,\,\, \text{with}\,\,\, ({\bf{z}}|{\bf{q}})=0.
		\end{equation}
		
		\noindent\textbf{Claim:} $c_{{\bf{r}}, {\bf{s}}}=c$ for all ${\bf{r}}, {\bf{s}} \in \Z^n$.\vspace{0.1cm}\\
		First, let ${\bf q}\notin \Q{\bf s}$ and choose ${\bf z}\in \C^n$ such that $({\bf{z}}|{\bf{q}})=0$ but $({\bf{z}}|{\bf{s}})\neq 0$. Taking ${\bf{r}}= j{\bf{s}}$, where $j \in \mathbb{Q}$ and $j{\bf{s}} \in \Z^n$ in \eqref{a11}, we obtain
		\begin{equation}\label{a12}
			j\,c_{{\bf{q}}+j{\bf{s}}, {\bf{s}}}+c_{j{\bf{s}}, {\bf{q}}+ {\bf{s}}} =  (j+1)c_{j{\bf{s}}, {\bf{s}}} \quad \quad \forall\, {\bf{q}} \notin \mathbb{Q}\,{\bf{s}}.
		\end{equation}
		Putting $j=1$ in \eqref{a12}, we get
		\begin{equation*}
			c_{{\bf{q}}+{\bf{s}}, {\bf{s}}}=c_{{\bf{s}}, {\bf{s}}}, \quad \quad \forall\, {\bf{q}} \notin \mathbb{Q}\,{\bf{s}}.
		\end{equation*}
		As $({\bf{q}}-{\bf{s}}) \notin \mathbb{Q}\,{\bf{s}}$, replacing ${\bf{q}}$ by ${\bf{q}}-{\bf{s}}$, we obtain $c_{{\bf{q}}, {\bf{s}}}=c_{{\bf{s}}, {\bf{s}}}$ for all ${\bf{q}} \notin \mathbb{Q}\,{\bf{s}}$.
		Interchanging ${\bf{q}}$ and ${\bf{s}}$, we have
		\begin{equation}\label{a14}
			c_{{\bf{q}}, {\bf{s}}}=c_{{\bf{s}}, {\bf{s}}}=c_{{\bf{q}}, {\bf{q}}}, \quad \quad \forall\, {\bf{q}} \notin \mathbb{Q}\,{\bf{s}}.
		\end{equation}
		Also, taking ${\bf r}=-({\bf q}+{\bf s})$ in \eqref{a11}, we get
		\begin{equation}\label{a13}
			\lambda_{-{\bf{s}}, {\bf{s}}}=\lambda_{-({\bf{q}}+{\bf{s}}), ({\bf{q}}+{\bf{s}})} \quad \quad \forall\, {\bf{q}} \notin \mathbb{Q}\,{\bf{s}}.
		\end{equation}
		Again, replacing  ${\bf{q}}$ by ${\bf{q}}-{\bf{s}}$ in \eqref{a13}, we get 
		\begin{equation}\label{a51}
			c_{-{\bf{s}}, {\bf{s}}}=c_{-{\bf{q}}, {\bf{q}}}, \quad \quad \forall\, {\bf{q}} \notin \mathbb{Q}\,{\bf{s}}.
		\end{equation}
		Now for any given ${\bf{q}}, {\bf{s}} \in \Z^n \setminus \{\bf 0\}$, choose ${\bf{k}} \in \Z^n$ such that ${\bf{k}} \notin \text{Span}_{\C}\,\{{\bf{q}},{\bf{s}}\}$. Then \eqref{a14} and \eqref{a51} give
		\begin{equation}\label{a15}
			c_{{\bf{s}}, {\bf{s}}}=c_{{\bf{k}}, {\bf{k}}}=c_{{\bf{q}}, {\bf{q}}}, \quad \quad c_{-{\bf{s}}, {\bf{s}}}=c_{-{\bf{k}}, {\bf{k}}}=c_{-{\bf{q}}, {\bf{q}}} \quad \quad \forall {\bf{q}}, {\bf{s}} \in \Z^n \setminus \{\bf{0}\}. 
		\end{equation}
		Using \eqref{a14}, \eqref{a15}, we have
		\begin{equation}\label{a53}
			c_{{\bf{q}}+j{\bf{s}}, {\bf{s}}}=c_{{\bf{s}}, {\bf{s}}}, \quad \text{and} \quad c_{j{\bf{s}}, {\bf{q}}+ {\bf{s}}}=c_{j{\bf{s}}, j{\bf{s}}}=c_{{\bf{s}}, {\bf{s}}}, \quad \forall\, {\bf{q}} \notin \mathbb{Q}\,{\bf{s}}.
		\end{equation}
		Substituting \eqref{a53} into \eqref{a12}, we get
		\begin{equation}\label{a16}
			c_{j{\bf{s}}, {\bf{s}}}=c_{{\bf{s}}, {\bf{s}}} \quad \quad \forall \,j \in \mathbb{Q}\setminus \{-1\}\,\, \text{with}\,\, j{\bf{s}} \in \Z^n.
		\end{equation}
		Since we are assuming $c_{{\bf{e}_1}, {\bf{e}_1}}=c_{{\bf{e}_1}, {\bf{-e}}_1}$, by using \eqref{a14},\eqref{a15},\eqref{a16}, we finally have
		\begin{equation}
			c_{{\bf{q}}, {\bf{s}}}=c_{{\bf{s}}, {\bf{s}}} \quad \quad \forall\, {\bf{q}}, {\bf{s}} \in \Z^n \setminus \{\bf{0}\}.
		\end{equation}
		Now let $v_0 \in V_{\bf{0}}'$ and consider 
		\begin{equation}
			\begin{aligned}\label{a17}
				(t^{\bf{q}}\,t^{\bf{s}})(t^{-\bf{q}}\,t^{-\bf{s}})v_0=&c_{{\bf{q}}, {\bf{s}}}\,c_{-{\bf{q}}, -{\bf{s}}}\,t^{{\bf{q}}+{\bf{s}}}\,t^{-({\bf{q}}+{\bf{s}})}\,v_0\\
				=&c_{{\bf{q}}, {\bf{s}}}\,c_{-{\bf{q}}, -{\bf{s}}}\, c_{{\bf{q}}+{\bf{s}}, -({\bf{q}}+{\bf{s}})}\,t^{\bf{0}}\,v_0\\
				=&(c_{{\bf{s}}, {\bf{s}}})^3\,c\,v_0
			\end{aligned}
		\end{equation} 
		and \begin{equation}\label{a18}
			(t^{\bf{q}}\,t^{-\bf{q}})(t^{\bf{s}}\,t^{-\bf{s}})v_0= c_{{\bf{q}}, -{\bf{q}}}\,c_{{\bf{s}}, -{\bf{s}}}\,(t^{\bf{0}})^2\,v_0=(c_{{\bf{s}}, {\bf{s}}})^2\,(c)^2\,v_0
		\end{equation}
		Combining \eqref{a17} and \eqref{a18}, we obtain $c_{{\bf{q}}, {\bf{s}}}=c$ for all ${\bf{q}}, {\bf{s}} \in \Z^n$.
		Now it is easy to check that $W=\{v \in V: t^{\bf{q}}\,t^{\bf{s}}\,v= c\, t^{{\bf{q}}+{\bf{s}}}\,v, \,\, \forall {\bf{q}}, {\bf{s}} \in \Z^n \}$ is a nonzero $\EuScript{G}$-submodule. Therefore we are done, thanks to the irreducibility of $V$.
	\end{proof}
\end{theorem}	
\begin{theorem}\label{a90}
	Let $V$ be an irreducible cuspidal $(\mathcal{D}_n \ltimes \mathcal{A}_n)$-module such that $\mathcal{A}_n'$ acts nontrivially on $V$. Then
	$V$ is isomorphic to $V(\lambda, \boldsymbol{\alpha}, c)$
	for some dominant integral weight $\lambda$ of $\mathfrak{sl}_n$,  $\boldsymbol{\alpha} \in \C^n$, and $c \in \C^{*}$.
	\begin{proof}
		By Theorem \ref{a19}, it follows that the $\mathcal{A}_n$-action on $V$ is associative. Hence $V$ satisfies the third condition in the definition of jet modules. Moreover, since each $t^{\bf m}$ acts injectively on $V$, it follows that $V$ is a free $\mathcal{A}_n$-module of finite rank. Therefore, $V$ is a jet module. Applying Theorem 5.2 of \cite{bta}, we conclude that
		$V \cong V(\lambda, \boldsymbol{\alpha}, c)$
		for some dominant integral weight $\lambda$ of $\mathfrak{sl}_n$,  $\boldsymbol{\alpha} \in \C^n$, and $c \in \C^{*}$.
	\end{proof}
\end{theorem}

\begin{cor}
	Let $V$ be an irreducible cuspidal module over $\EuScript{G}$. Then we have
	\begin{enumerate}
		\item $P(V)= \mu + \Z^n$ for some $\mu \in \mathcal{H}^*$.
		\item $\dim\, V_{\mu_1}=\dim\, V_{\mu_2}$ for all $\mu_1, \mu_2 \in P(V)$.
	\end{enumerate}
\end{cor}
\section{Generalized Highest weight modules}\label{a95}
In this section, we define generalized highest weight modules for the Lie algebra $\EuScript{G}$ and show that every irreducible Harish-Chandra module over $\EuScript{G}$ is either cuspidal or a generalized highest weight module. We begin by fixing some notation.

\vspace{0.3cm}
\noindent\textbf{Notation.}
For ${\bf m}, {\bf k} \in \Z^n$, we write ${\bf m} \geq {\bf k}$ if $m_i \geq k_i$ for all $1 \leq i \leq n$. For integers $p,q$, we set $[p,q]= \{k \in \Z \mid p \leq k \leq q\}$.
The intervals $(\infty, p]$ and $[q, \infty)$ are defined similarly.

\begin{definition}
	A $\EuScript{G}$-module $V$ is called a generalized highest weight (GHW) module if there exists a nonzero vector $v \in V_{\lambda}$ and $k \in \N$ such that 
	\begin{center}
		$\EuScript{G}_{{\bf{m}}}.v=0$ \,\,\,\, for all \, ${\bf{m}} \geq (k,k,\ldots,k)$.
	\end{center} 
	Such a vector $v$ is called a GHW vector with GHW $\lambda$.
	
Alternatively, we have the following equivalent definition of a GHW module. Suppose there exists a $\Z$-basis $B = \{{\bm{\alpha}}_1, {\bm{\alpha}}_2, \ldots, {\bm{\alpha}}_n\}$ of $\Z^n$ and a nonzero vector $v \in V_{\lambda}$ such that
\begin{center}
	$\EuScript{G}_{\bf m} v = 0 \quad \text{for all } {\bf 0} \neq {\bf m} \in \sum_{i=1}^n \Z_+ {\bm{\alpha}}_i$.
\end{center}
Then $V$ is called a GHW module with GHW $\lambda$ with respect to the basis $B$.

\end{definition}
Using induction, one can easily prove the following lemma.
\begin{lemma}\label{a25}
	Let $V$ be an irreducible Harish-Chandra module over $\EuScript{G}$. If there exists a $\Z$-basis $\{{\bm{\al}}_1,\bm{\al}_2,\ldots,\bm{\al}_n\}$ of $\Z^n$ and a weight vector $v \in V$ such that $\EuScript{G}_{{\bm{\al}}_i}v=0$ for all $i=1,2,\ldots,n$, then $V$ is a GHW module up to a change of coordinates.
\end{lemma}
\begin{theorem}\label{a75}
	Every nontrivial irreducible Harish-Chandra module over $\EuScript{G}$ is either cuspidal or a GHW module.
	\begin{proof}
		Let $V$ be a nontrivial irreducible Harish-Chandra $\EuScript{G}$-module and assume that $V$ is not a GHW module. We will show that $V$ must be cuspidal. Since $V$ is irreducible,
		$V=\bigoplus_{{\bf m}\in\Z^n}V_{\lambda+{\bf m}}$ for some $\lambda \in \mathcal{H}^*$. Let ${\bf m}=(m_1,m_2,\ldots,m_n)$ and set $V_{\bf m}:=V_{\lambda+{\bf m}}$. For $1 \leq i < j \leq n$, consider the linear maps
		\begin{equation*}
			\begin{aligned}
				t^{-m_1{\bf{e}}_1+{\bf e}_2}: V_{\bf{m}} \rightarrow V_{(0,m_2+1, m_3,\ldots, m_n)}, \quad \quad D_{ij}(-m_1{\bf{e}}_1+{\bf e}_2): V_{\bf{m}} \rightarrow V_{(0,m_2+1, m_3,\ldots, m_n)}\\
				t^{(1-m_1){\bf{e}}_1+{\bf e}_2}: V_{\bf{m}} \rightarrow V_{(1,m_2+1, m_3,\ldots, m_n)}, \quad D_{ij}((1-m_1){\bf{e}}_1+{\bf e}_2): V_{\bf{m}} \rightarrow V_{(1,m_2+1, m_3,\ldots, m_n)}\\
				t^{-m_1{\bf{e}}_1+{\bf e}_k}: V_{\bf{m}} \rightarrow V_{(0,m_2,\ldots ,m_k+1,\ldots, m_n)}, \quad \quad D_{ij}(-m_1{\bf{e}}_1+{\bf e}_k): V_{\bf{m}} \rightarrow V_{(0,m_2,\ldots ,m_k+1,\ldots, m_n)}
			\end{aligned}
		\end{equation*}
		for each $k=3,\ldots,n$. Since $\{-m_1{\bf{e}}_1+{\bf e}_2,  (1-m_1){\bf{e}}_1+{\bf e}_2, -m_1{\bf{e}}_1+{\bf e}_k: 3 \leq k \leq n\}$ is a $\Z$-basis for $\Z^n$ and $V$ is not a GHW module, Lemma \ref{a25} implies that the intersection of the kernels of the above maps is zero. Therefore, we have
		\begin{equation*}
			\textrm{dim}\, V_{\bf{m}} \,\leq (\frac{n^2-n}{2}+1)\, \Big(\sum_{k=0}^{1}\textrm{dim}\,V_{(k,m_2+1, m_3,\ldots, m_n)}+ \sum_{k=3}^{n}\textrm{dim}\,V_{(0,m_2, m_3,\ldots,m_k+1,\ldots, m_n)}\Big).
		\end{equation*}
		Now let ${\bf{m}}':=(0,m_2+1, m_3,\ldots, m_n)$ and concentrate on $V_{{\bf{m}}'}$. By the same argument, using the maps
		\begin{equation*}
			\begin{aligned}
				&t^{{\bf{e}}_1-m_2{\bf e}_2}: V_{\bf{m}'} \rightarrow V_{(1,1, m_3,\ldots, m_n)}, \qquad \qquad \qquad D_{ij}({\bf{e}}_1-m_2{\bf e}_2): V_{\bf{m}'} \rightarrow V_{(1,1, m_3,\ldots, m_n)}\\
				&t^{{\bf{e}}_1+(1-m_2){\bf e}_2}: V_{\bf{m}'} \rightarrow V_{(1,2, m_3,\ldots, m_n)}, \qquad \qquad D_{ij}({\bf{e}}_1+(1-m_2){\bf e}_2): V_{\bf{m}'} \rightarrow V_{(1,2, m_3,\ldots, m_n)}\\
				&t^{-m_2{\bf{e}}_2+{\bf e}_k}: V_{\bf{m}'} \rightarrow V_{(0,1, m_3,\ldots, m_k+1,\ldots,m_n)}, \quad \quad D_{ij}(-m_2{\bf{e}}_2+{\bf e}_k): V_{\bf{m}'} \rightarrow V_{(0,1, m_3,\ldots, m_k+1, \ldots, m_n)}
			\end{aligned}
		\end{equation*}
		for each $k=3,\ldots,n$, we obtain
		\begin{equation*}
			\textrm{dim}\, V_{\bf{m}'} \leq (\frac{n^2-n}{2}+1)\, \Big(\sum_{k=1}^{2}\textrm{dim}\,V_{(1,k, m_3,\ldots, m_n)}+ \sum_{k=3}^{n}\textrm{dim}\,V_{(0,1, m_3,\ldots,m_k+1,\ldots, m_n)}\Big).
		\end{equation*}
		Repeating this process finitely many times, we see that $\dim V_{\bf m}$ is bounded above by a fixed constant independent of ${\bf m}$ and that completes the proof.
	\end{proof}
\end{theorem}

We now discuss some important properties of GHW modules. The following lemma can be proved by arguments similar to those in Lemmas 3.3 and 3.4 of \cite{maz}.
\begin{lemma}\label{a23}
	Let $V$ be an irreducible GHW module over $\EuScript{G}$. Then the following hold:
	\begin{enumerate}
		\item Every nonzero vector of $V$ is a GHW vector.\vspace{0.1cm}
		\item For any nonzero vector $v \in V$ and ${\bf m} \in \N^n$, we have $\EuScript{G}_{-{\bf m}}v \neq \{0\}$.
	\end{enumerate}
\end{lemma}

\begin{lemma}\label{a76}
	Let $V$ be a nontrivial irreducible GHW module over $\EuScript{G}$. Then, for every $\mu \in P(V)$ and ${\bf m} \in \N^n$, there exists $l \in \Z_+$ such that
	\[
	\{k \in \Z : \mu + k{\bf m} \in P(V)\} = (-\infty, l].
	\]
	\begin{proof}
		Set $J := \{k \in \Z : \mu + k{\bf m} \in P(V)\}$.
		By Lemma \ref{a23}, either $J = (-\infty, l]$ for some $l \in \Z_+$ or $J = \Z$.  Suppose that $J = \Z$. Then $\mu + i{\bf m} \in P(V)$ for every $i \in \Z$. For each $i \in \Z$, choose $v_i \in V_{\mu + i{\bf m}}$. Since every nonzero vector of $V$ is a GHW vector, there exists $p_i \in \N$ such that
		\begin{center}
			$\EuScript{G}_{{\bf{q}}}\,v_i=0$ \,\,\,\,for all\,\,\, ${\bf{q}}> (p_i, p_i, \ldots, p_i)$.
		\end{center}
		Now choose a subsequence $(x_n)$ of $\N$ such that $x_{i+1} > x_i + p_{x_i} + 2$. For each fixed $i \in \N$, there exists $q_i \in \N$ such that
		\begin{equation}\label{a72}
			\EuScript{G}_{y{\bf m}+{\bf e}_1}v_{x_i} = 0 \quad \text{for all } y \geq q_i.
		\end{equation}
		On the other hand, Lemma \ref{a23}(2) implies that
		\begin{equation}\label{a73}
			\EuScript{G}_{y{\bf m}+{\bf e}_1}v_{x_i} \neq \{0\} \quad \text{for all } y < -1.
		\end{equation}
		Therefore, by \eqref{a70}, \eqref{a72}, \eqref{a73}, there exist $y_i \geq -2$ and $2 \leq i' \leq n$ such that
		\begin{equation*}
			\{D_{1i'}(y_i{\bf m}+{\bf e}_1)\,v_{x_i}, \, t^{y_i{\bf m}+{\bf e}_1}\,v_{x_i}\} \neq \{0\}\qquad \text{and} \qquad \EuScript{G}_{y{\bf m}+{\bf e}_1}v_{x_i} = 0, \quad \forall \,\,y > y_i.
		\end{equation*}
		Now define
		\begin{equation*}
			A_i =
			\begin{cases}
				t^{-x_i{\bf m}}, & \text{if}\,\,\, t^{y_i{\bf m}+{\bf e}_1}v_{x_i} \neq 0,\\[0.1cm]
				D_{1i'}(-x_i{\bf m}), & \text{if}\,\,\, t^{y_i{\bf m}+{\bf e}_1}v_{x_i} = 0.
			\end{cases}
		\end{equation*}
		We claim that the set $\{A_i v_{x_i} : i \in \N\}$ is linearly independent. Fix $r \in \N$, and suppose that
		$\sum_{i=1}^{r} c_i A_i v_{x_i} = 0$. Consider
		\begin{equation}\label{a71}
			0=D_{1r'}((x_r+y_r){\bf m}+{\bf e}_1) \Big(\sum_{i=1}^r c_iA_i\,v_{x_i}\Big)=\sum_{i=1}^r c_i\,[D_{1r'}((x_r+y_r){\bf m}+{\bf e}_1),\, A_i]\,v_{x_i}
		\end{equation}
		Note that
		\begin{equation*}
			\begin{aligned}
				[D_{1r'}((x_r+y_r){\bf m}+{\bf e}_1), t^{-x_r{\bf m}}]
				=& -x_r m_{r'}\, t^{y_r{\bf m}+{\bf e}_1},\\
				[D_{1r'}((x_r+y_r){\bf m}+{\bf e}_1), D_{1r'}(-x_r{\bf m})]
				=& -x_r m_{r'}\, D_{1r'}(y_r{\bf m}+{\bf e}_1),
			\end{aligned}
		\end{equation*}
		and for $i \neq r$, we have
		\begin{equation*}
			[D_{1r'}((x_r+y_r){\bf m}+{\bf e}_1), A_i] \in \EuScript{G}_{(x_r+y_r-x_i){\bf m}+{\bf e}_1}.
		\end{equation*}
		Since $x_r + y_r - x_i > p_{x_i}$, it follows that
		$[D_{1r'}((x_r+y_r){\bf m}+{\bf e}_1), A_i]v_{x_i} = 0$. Therefore, from \eqref{a71} and the choice of $y_r$, we obtain $c_r = 0$. Similarly, we can show that $c_{r-1}=\cdots=c_1=0$. This implies that $\dim V_{\mu}=\infty$, which is a contradiction. Therefore $J \neq \Z$, and consequently
		$J = (-\infty, l]$ for some $l \in \Z_+$.
	\end{proof}
\end{lemma}
\begin{lemma}\label{a30}
	Let $V$ be an irreducible Harish-Chandra GHW module over $\EuScript{G}$ which is not cuspidal. Then, after a suitable change of coordinates, the following statements hold:
	\begin{enumerate}
		\item $V$ is a GHW module with GHW $\lambda$ over $\EuScript{G}$ with respect to the standard basis of $\Z^n$.\vspace{0.1cm}
		\item $\lambda + {\bf m} \notin P(V)$ for all ${\bf 0} \neq {\bf m} \in \Z_+^n$.\vspace{0.1cm}
		\item $\lambda - {\bf m} \in P(V)$ for all ${\bf m} \in \Z_+^n$.\vspace{0.1cm}
		\item For any ${\bf m}, {\bf p} \in \Z^n$ satisfying ${\bf m} \leq {\bf p}$, the condition $\lambda + {\bf m} \notin P(V)$ implies $\lambda + {\bf p} \notin P(V)$.\vspace{0.1cm}
		\item For any ${\bf m}, {\bf p} \in \Z^n$ satisfying ${\bf m} \leq {\bf p}$, the condition $\lambda + {\bf p} \in P(V)$ implies $\lambda + {\bf m} \in P(V)$.\vspace{0.1cm}
		\item For any ${\bf 0} \neq {\bf m} \in \Z_+^n$ and ${\bf p} \in \Z^n$, we have
		\begin{equation*}
			\{k \in \Z : \lambda + {\bf l} + k{\bf m} \in P(V)\} = (-\infty, q], \,\,\,\, \text{for some}\,\, q \in \Z.
		\end{equation*}
	\end{enumerate}
	\begin{proof}
		By Theorem \ref{a75}, $V$ is a GHW module. Consequently, statements (1)-(5) can be proved by arguments analogous to those used in Lemma 3.6 of \cite{maz}. Statement (6) follows from Lemma \ref{a23} and Lemma \ref{a76}.
	\end{proof}
\end{lemma}
\section{Classification of irreducible GHW modules}\label{a96}
		In this section, we define highest weight modules for $\EuScript{G}$ with respect to a suitable triangular decomposition. Let
		$\EuScript{G} = \bigoplus_{{\bf m} \in \Z^n} \EuScript{G}_{\bf m}$
		be the graded decomposition given by \eqref{a24}. Let $M$ be a subgroup of $\Z^n$ and let $\boldsymbol{\beta} \in \Z^n$ be such that
		$\Z^n = M \oplus \Z \boldsymbol{\beta}$. With respect to this decomposition of $\Z^n$, we define a $\Z$-grading on $\EuScript{G}$ by
		\[
		\EuScript{G} = \bigoplus_{r \in \Z} \EuScript{G}_{_{M+r\boldsymbol{\beta}}},\,\,\,\, \text{where}\,\,\, \EuScript{G}_{_{M+r\boldsymbol{\beta}}}=\bigoplus_{{\bf m} \in M} \EuScript{G}_{{\bf m}+r\boldsymbol{\beta}}.
		\]
		This induces a triangular decomposition of $\EuScript{G}$, namely
		$\EuScript{G} = \EuScript{G}^{-}_M \oplus \EuScript{G}_M \oplus \EuScript{G}^{+}_M$, where
		\[
		\EuScript{G}^{-}_M = \bigoplus_{\substack{{\bf m} \in M \\ r \in \N}} \EuScript{G}_{{\bf m}-r\boldsymbol{\beta}}, 
		\qquad
		\EuScript{G}_M = \bigoplus_{{\bf m} \in M} \EuScript{G}_{\bf m},
		\qquad
		\EuScript{G}^{+}_M = \bigoplus_{\substack{{\bf m} \in M \\ r \in \N}} \EuScript{G}_{{\bf m}+r\boldsymbol{\beta}}.
		\]
		
		Using this triangular decomposition, we define highest weight modules for $\EuScript{G}$. Let $X$ be an irreducible module over $\EuScript{G}_M$. We extend $X$ to a $(\EuScript{G}_{_{M}} \oplus \EuScript{G}^{+}_M)$-module by letting $\EuScript{G}^{+}_M$ act trivially on $X$. The corresponding generalized Verma module for $\La$ is defined by
		\begin{center}
			$\mathbb{M}_{\EuScript{G}}(X,\boldsymbol{\beta}, M)= \textrm{Ind}_{\EuScript{G}_{_{M}} \oplus \EuScript{G}^{+}_M}^{\EuScript{G}} X= \mathcal{U}(\EuScript{G}) \otimes_{\EuScript{G}_{_{M}} \oplus \EuScript{G}^{+}_M}  X.$ 
		\end{center}
		As a vector space, $\mathbb{M}_{\EuScript{G}}(X,\boldsymbol{\beta}, M) \cong \mathcal{U}(\EuScript{G}^{-}_M) \otimes_{\C}  X$. By standard arguments, $\mathbb{M}_{\EuScript{G}}(X,\boldsymbol{\beta}, M)$ has a unique maximal submodule that intersects $X$ trivially. We denote the corresponding unique irreducible quotient by $\mathbb{L}_{\EuScript{G}}(X,\boldsymbol{\beta}, M)$. For simplicity, we write
		$\mathbb{L}(X,\boldsymbol{\beta}, M) := \mathbb{L}_{\EuScript{G}}(X,\boldsymbol{\beta}, M)$.
		
		Since $\EuScript{G}$ is $\Z^n$-extragraded, by a result of Billig and Zhao (\cite{bz}, Theorem 1.5), it is known that if $X$ is a uniformly bounded $\Z^{n-1}$-graded exp-polynomial $\uptau_{M}$-module, then 
		$\mathbb{L}(X,\boldsymbol{\beta}, M)$ becomes a Harish-Chandra module for $\EuScript{G}$.  We refer to Section 1 of \cite{bz} for more details on $\Z^n$-extragraded Lie algebras and $\Z^n$-graded exp-polynomial modules.
		
		\vspace{0.3cm}
		From now on, $V$ always denotes an irreducible nontrivial GHW module over $\EuScript{G}$ with GHW vector $v_{\Lambda}$ of weight $\Lambda$ with respect to the standard basis of $\C^n$. With this assumption, our main objective is to prove the following theorem.
		\begin{theorem}\label{a28}
			Let $V$ be an irreducible GHW module. Then $V \cong \mathbb{L}(X,\boldsymbol{\beta}, M)$ for some subgroup $M$ of $\Z^n$, $\boldsymbol{\beta} \in (\Z^n)^*$ such that $\Z^n= M \oplus \Z \boldsymbol{\beta}$, and irreducible $\EuScript{G}_M$-module $X$.
		\end{theorem}
		We need several preparatory results to prove this theorem. We begin with the following proposition, which will play a crucial role in this regard.
		
		\begin{prop}\label{a40}
			Let ${\bf{s}}, {\bf{k}} \in \Z^n$ such that $k_1, k_2,\ldots, k_n$ are relatively prime. Suppose there exists ${\bf{p}} \in \N^n$ such that 
			\begin{equation}
				\{ \Lambda + \sum_{i=1}^{n} s_i{\bf{e}_i} + \sum_{i=1}^{n} r_ip_i{\bf{e}_i} \in P(V): {\bf{r}} \in \Z^n,\,\, \sum_{i=1}^{n} k_ip_ir_i=0 \}= \emptyset.
			\end{equation}
			Then $V$ is isomorphic to $\mathbb{L}(X,\boldsymbol{\beta}, M)$ for some $X, \boldsymbol{\beta}$, and $M$ defined above.
			\begin{proof}
				Set $M=\{{\bf{l}} \in \Z^n: ({\bf{k}}|{\bf{l}})=0 \}$. Now, by following the proof of Lemma 3.3 (Claim 1) of \cite{lz}, we can show that there exists a unique integer $m_0$ such that
				\begin{equation}\label{a26}
					\begin{aligned}
						&\{ \Lambda + {\bf{l}} \in P(V): {\bf{l}} \in \Z^n, ({\bf{k}}|{\bf{l}}) \geq m_0 \}= \emptyset \,\,\,\,\,\,\, \text{and},\\
						Q=\,&\{ \Lambda + {\bf{l}} \in P(V): {\bf{l}} \in \Z^n, ({\bf{k}}|{\bf{l}})= m_0-1 \} \neq  \emptyset.			
					\end{aligned}
				\end{equation}
				Since $k_1, k_2, \ldots, k_n$ are relatively prime, there exist $\beta_1, \beta_2, \ldots, \beta_n \in \Z$ such that $\sum_{i=1}^{n} k_i \beta_i = 1$. Let $\boldsymbol{\beta}=(\beta_1,\ldots,\beta_n)$. If ${\bf m}\in\Z^n$ satisfies $({\bf k}\mid {\bf m})=j$, then ${\bf m}-j\boldsymbol{\beta}\in M$, and that gives $\Z^n=M\oplus \Z\boldsymbol{\beta}$. Moreover, \eqref{a26} implies that
				\begin{equation}\label{a27}
					(Q+M+\N\boldsymbol{\beta})\cap P(V)=\emptyset.
				\end{equation}
				Fix $\mu \in Q$. Then it is easy to check that 
				\begin{equation*}
					Q = (\mu + M) \cap P(V).
				\end{equation*}
				Let $\EuScript{G}=\EuScript{G}^{-}_M\oplus \EuScript{G}_M\oplus \EuScript{G}^{+}_M$ be the triangular decomposition with respect to $\Z^n=M\oplus \Z\boldsymbol{\beta}$. Now define $X= \oplus_{{\bf{m}} \in M} V_{\mu+{\bf{m}}}$, which is a module over $\EuScript{G}_M$. By \eqref{a27}, we have $\EuScript{G}^{+}_M \cdot X = 0$. Since $V$ is irreducible as a $\EuScript{G}$-module, an application of the PBW theorem together with weight considerations shows that $X$ is irreducible as a $\EuScript{G}_M$-module. Therefore,
				\[
				V \cong \mathbb{L}(X,\boldsymbol{\beta}, M).
				\]
			\end{proof}
		\end{prop}
We now deal with the case $n=2$ and prove Theorem \ref{a27} in this setting. So our next few results up to Theorem \ref{a36} are devoted to the special case $n=2$. For any ${\bf r}=(r_1,r_2) \in \Z^2$, set $\bar{\bf r}=(r_2,-r_1)$. Observe that, for $n=2$, the condition $({\bf u}|{\bf v})=0$ implies that ${\bf u}=\lambda \bar{\bf v}$ for some $\lambda\in\C$. For ${\bf b}\in\Z^2$, set
$d({\bf b})=D(\bar{\bf b},{\bf b})$.
Thus $\EuScript{G}$ is spanned by
$\{d({\bf r}),\,t^{\bf s},\,d_1,\,d_2 : {\bf r}\in\Z^2\setminus\{(0,0)\},\ {\bf s}\in\Z^2\}$.
\begin{prop}\label{a80}
	Suppose there exists a $\Z$-basis $\{{\bf b}_1,{\bf b}_2\}$ of $\Z^2$ and $\lambda\in P(V)$ such that
	\begin{equation*}
		(\lambda+\Z{\bf b}_1+\N{\bf b}_2)\cap P(V)=\emptyset.
	\end{equation*}
	Then $V \cong \mathbb{L}(X,{\bf b}_2,\Z{\bf b}_1)$ for some irreducible $\EuScript{G}_{\Z{\bf b}_1}$-module $X$.
	\begin{proof}
		Set $X=\bigoplus_{r\in\Z}V_{\lambda+r{\bf b}_1}$,\, $M=\Z{\bf b}_1$. By the hypothesis, we have $\EuScript{G}_M^{+}\cdot X=0$. We first show that $X$ is a cuspidal $\EuScript{G}_M$-module. For each $r\in\Z$, define
		\begin{equation*}
			\psi_r:V_{\lambda+r{\bf b}_1}\to V_{\lambda-{\bf b}_2}\oplus V_{\lambda-{\bf b}_2}\quad \text{by} \quad
			\psi_r(v)=\big(d(-r{\bf b}_1-{\bf b}_2)v,\ t^{-r{\bf b}_1-{\bf b}_2}v\big).
		\end{equation*}
		Since $\Z^2$ is generated by $\{(r+1){\bf b}_1+{\bf b}_2,\,-r{\bf b}_1-{\bf b}_2\}$, the Lie algebra $\EuScript{G}$ is generated by
		$\{\EuScript{G}_{(r+1){\bf b}_1+{\bf b}_2},\ \EuScript{G}_{-r{\bf b}_1-{\bf b}_2}\}$. Hence $\psi_r$ is injective by the irreducibility of $V$, which implies that $X$ is cuspidal over $\EuScript{G}_M$. 
		
		By an application of the PBW theorem, we see that $X$ is an irreducible $\EuScript{G}_M$-module. Moreover, by the construction of $\mathbb{M}_{\EuScript{G}}(X,{\bf b}_2,\Z{\bf b}_1)$ and an application of the PBW theorem, there exists a surjective homomorphism
		from $\mathbb{M}_{\EuScript{G}}(X,{\bf b}_2,\Z{\bf b}_1)$ to $V$. Since $V$ is irreducible, we conclude that
		$V \cong \mathbb{L}(X,{\bf b}_2,\Z{\bf b}_1)$.
	\end{proof}
\end{prop}

We now state some useful results in the following lemma. Their proofs follow by using Lemma \ref{a30}, Proposition \ref{a40}, and Proposition \ref{a80}, in a similar way as in \cite{lz,lt,ts}.
\begin{lemma}\label{a29}
	In each of the following cases, $V$ is isomorphic to a highest weight module for $\La$.
	\begin{enumerate}
		\item There exist $(s_1,s_2)\in\Z^2$, $(r_1,r_2)\in\Z^2\setminus\{(0,0)\}$, and $l,m\in\Z$ such that
		\[
		(-\infty,l)\cup(m,\infty)\subseteq \{p\in\Z:\Lambda+(s_1,s_2)+p\,(r_1,r_2)\in P(V)\}.
		\]
		
		\item There exist $(s_1,s_2)\in\Z^2$ and $(r_1,r_2)\in\Z^2\setminus\{(0,0)\}$ such that
		\[
		\{p\in\Z:\Lambda+(s_1,s_2)+p\,(r_1,r_2)\in P(V)\}=\emptyset.
		\]
		
		\item There exist $(s_1,s_2),(r_1,r_2)\in\Z^2$ and $p_1,p_2,p_3\in\Z$ with $p_1<p_2<p_3$ such that
		\begin{equation*}
			\begin{aligned}
				&\Lambda+(s_1,s_2)+p_1(r_1,r_2)\notin P(V),\\
				&\Lambda+(s_1,s_2)+p_2(r_1,r_2)\in P(V), \,\,\text{and}\\
				&\Lambda+(s_1,s_2)+p_3(r_1,r_2)\notin P(V).
			\end{aligned}
		\end{equation*}
	\end{enumerate}
\end{lemma}
\begin{lemma}\label{a32}
	Let $\mu\in P(V)$. Suppose there exist integers $i>0$, $j<0$, and ${\bf 0} \neq {\bf s}=(s_1,s_2)\in\Z^2 \setminus \{(0,0)\}$ such that
	\[
	d(i{\bf s})v=t^{i{\bf s}}v=0=d(j{\bf s})v=t^{j{\bf s}}v
	\qquad \text{for some } v\in V_\mu.
	\]
	Then $V$ is isomorphic to a highest weight module for $\La$.
	\begin{proof}
		Let $c=\text{gcd}\,(s_1,s_2)$, then we can write ${\bf{s}}=c\,(p_1, p_2)$, where $\text{gcd}(p_1, p_2)=1$. Hence there exist integers $q_1,q_2$ such that
		$p_1q_2-p_2q_1=1$. Set ${\bf p}=(p_1,p_2), \,\, {\bf q}=(q_1,q_2)$.
		Then $\{{\bf p},{\bf q}\}$ is a $\Z$-basis of $\Z^2$. Now assume for the contrary, $V$ is not a highest weight module. Then Lemma \ref{a29} implies that, for every $0\neq l\in\Z$, there exists $r_l\in\Z$ such that
		\begin{center}
			$B_l=\{k \in \Z: \mu +l{\bf{q}}+k{\bf{p}} \in P(V)\}=(-\infty, r_l]$ or $[r_l, \infty)$.
		\end{center} 
		First assume that $B_l=(-\infty,r_l]$. Then there exists $s_l\in\N$ such that
		$\EuScript{G}_{\,l{\bf q}-(jcs_l\pm1){\bf p}}\,v=0$.	Hence by using the relation $\EuScript{G}_{jc{\bf p}}v=0$ together with the commutator relations in $\EuScript{G}$, we deduce that $\EuScript{G}_{\,l{\bf q}\pm{\bf p}}v=0$. In particular,
		\begin{equation*}
			\EuScript{G}_{\pm({\bf q}+{\bf p})}v=0=\EuScript{G}_{\pm(2{\bf q}+{\bf p})}v.
		\end{equation*}
		Since $\{{\bf q}+{\bf p},\,2{\bf q}+{\bf p}\}$ is a $\Z$-basis of $\Z^2$, the irreducibility of $V$ implies that $V$ is trivial, which is a contradiction. Similarly, if $B_l=[r_l,\infty)$, then using $\EuScript{G}_{i{\bf s}}v=0$, we again obtain a contradiction. This completes the proof.
	\end{proof}
\end{lemma}
\vspace{0.75cm}
From now on, we assume that $V$ is not a highest weight module. Then by using Lemma \ref{a29}, it is easy to see that for any $(s_1,s_2) \in \Z^2$ and $(0,0) \neq (r_1,r_2) \in \Z^2$, there exists $l \in \Z$ such that
\begin{equation}\label{a31}
	\{p \in \Z: \Lambda+(s_1,s_2)+p(r_1,r_2) \in P(V)\}=(-\infty, l]\,\, \text{or}\,\, [l, \infty).
\end{equation}
Hence, Lemma \ref{a30} implies that, for every $i \in \N$, there exist $l_i, m_i \in \Z_+$ such that
\begin{equation*}
	\{p \in \Z : \Lambda + (-i,p) \in P(V)\} = (-\infty,l_i],
	\qquad
	\{p \in \Z : \Lambda + (p,-i) \in P(V)\} = (-\infty,m_i].
\end{equation*}
We can now derive the following statements by using arguments similar to those used in Theorem 3.7 of \cite{lz}.

\smallskip
\noindent
(S1) The following limits exist:
\begin{equation*}
	\lim_{i \to \infty} \frac{l_i}{i} = \alpha,
	\qquad
	\lim_{i \to \infty} \frac{m_i}{i} = \beta.
\end{equation*}

\noindent
(S2) $\alpha$ and $\beta$ are positive irrational numbers satisfying $\alpha = \beta^{-1}$.

\noindent
(S3) Moreover, the order $<_{\alpha}$ is dense on $\Z^2$, i.e., for every $(p,q) >_{\alpha} (0,0)$, there exist infinitely many $(r,s) \in \Z^2$ such that
\[
(0,0) <_{\alpha} (r,s) <_{\alpha} (p,q).
\]

\noindent
(S4) If $\Lambda + (r,s) \in P(V)$, then $\Lambda + (p,q) \in P(V)$ for all $(p,q) \in \Z^2$ satisfying
$(p,q) <_{\alpha} (r,s)$.

Set
\[
\Z^2(+) := \{(p,q) \in \Z^2 : (p,q) >_{\alpha} (0,0)\},
\qquad
\Z^2(-) := \{(p,q) \in \Z^2 : (p,q) <_{\alpha} (0,0)\}.
\]
This order naturally induces a triangular decomposition
\begin{equation*}
	\EuScript{G} = \EuScript{G}^{+}_{<_{\alpha}} \oplus \EuScript{G}^{0}_{<_{\alpha}} \oplus \EuScript{G}^{-}_{<_{\alpha}}, \qquad \text{where}\qquad \EuScript{G}^{0}_{<_{\alpha}} = \C d_1 \oplus \C d_2 \oplus \C t^{\bf 0}.
\end{equation*}

\begin{lemma}\label{a35}
	Assume that $V$ is not a highest weight module, and let $\mu \in P(V)$. Then the following statements hold:
	\begin{enumerate}
		\item For every ${\bf{m}}=(m_1,m_2) \in \Z^2(+) $ and every $v \in V_{\mu}$, we have $\EuScript{G}_{-{\bf{m}}}v \neq \{0\}$.
		\medskip
		\item $(\mu + \Z^2(+)) \cap P(V) \neq \emptyset$.
	\end{enumerate}
	\begin{proof}
		(1) Suppose that $\EuScript{G}_{-{\bf m}}v = 0$ for some ${\bf m}=(m_1,m_2) \in \Z^2(+)$ and some $v \in V_{\mu}$. By \eqref{a31}, together with (S3) and (S4), it follows that, for every $(s_1,s_2) \in \Z^2$, there exists $l \in \Z$ such that
		\begin{equation*}
			\{p \in \Z : \Lambda + (s_1,s_2) + p(r_1,r_2) \in P(V)\} = (-\infty,l], \qquad \forall\, (r_1,r_2) \in \Z^2(+).
		\end{equation*}
		In particular, for the fixed ${\bf m}$, we may choose $p \in \N$ sufficiently large so that $\EuScript{G}_{p{\bf m}}v = 0$. Therefore, Lemma \ref{a32} implies that $V$ is a highest weight module, a contradiction. This proves (1).
		
		\smallskip
		(2) Suppose, on the contrary, that $(\mu + \Z^2(+)) \cap P(V) = \emptyset$. By the density of the order $<_{\alpha}$ from (S3), for any $k \in \N$, we can choose ${\bf b}=(r,s) \in \Z^2$ with $r \neq 0$ such that $-\frac{1}{4k} < r\alpha + s < 0$. Therefore, $-{\bf e}_2 <_{\alpha} i{\bf b} <_{\alpha} (0,0)$, for all $1 \leq i \leq 4k$. Fix $v \in V_{\mu}$. Since $t^{-{\bf b}}v = 0 = d(-{\bf b})v$, Lemma \ref{a32} implies that $t^{\bf b}v \neq 0$ or $d({\bf b})v \neq 0$. Moreover, using
		\[
		t^{-2{\bf b}}t^{\bf b}v = d(-2{\bf b})t^{\bf b}v = t^{-2{\bf b}}d({\bf b})v = d(-2{\bf b})d({\bf b})v = 0,
		\]
		and applying Lemma \ref{a32} again, we deduce that at least one of the vectors
		$t^{\bf b}t^{\bf b}v, \,t^{\bf b}d({\bf b})v$, $d({\bf b})d({\bf b})v$
		is nonzero. In each case, we shall obtain a contradiction by showing that $V_{\mu-{\bf e}_2}$ is infinite-dimensional. For each $1 \leq i \leq k$, define
		\[
		{\bf p}_i := -(2i-1){\bf b} \in \Z^2(+), \qquad {\bf q}_i := {\bf e}_2 - {\bf p}_i \in \Z^2(+).
		\]
		Assume first that $t^{\bf b}t^{\bf b}v \neq 0$. We claim that
		$\{t^{-{\bf p}_i}t^{-{\bf q}_i}v : 1 \leq i \leq k\} \subseteq V_{\mu-{\bf e}_2}$
		is linearly independent. Set ${\bf{m}}={\bf{e}}_2+2{\bf{b}} \in \Z^2(+)$. Consider
		\begin{equation*}
			\begin{aligned}
				0&=\sum_{i=1}^{k} \lambda_i\, t^{-{\bf{p}}_i}\,t^{-{\bf{q}}_i}v=d( {\bf{m}})\Big(\sum_{i=1}^{k} \lambda_i\, t^{-{\bf{p}}_i}\,t^{-{\bf{q}}_i}\Big)v=\sum_{i=1}^{k} \lambda_i\,[d({\bf{m}}),\, t^{-{\bf{p}}_i}\,t^{-{\bf{q}}_i} ]\,v\\
				&=-\sum_{i=1}^{k} \lambda_i\,\Big( (\bar{\bf{m}}|{\bf{p}}_i)\,t^{{\bf{m}}-{\bf{p}}_i} t^{-{\bf{q}}_i} + (\bar{\bf{m}}|{\bf{q}}_i)\,t^{-{\bf{p}}_i}\,t^{{\bf{m}}-{\bf{q}}_i} \Big)v\\
				&=-\lambda_1 (\bar{\bf{m}}|{\bf{q}}_1)\, t^{-{\bf{p}}_1}\,t^{{\bf{m}}-{\bf{q}}_1}v=-r\lambda_1\,t^{\bf{b}}t^{\bf{b}}v
			\end{aligned}
		\end{equation*}
		Hence, $\lambda_1 = 0$. Here we use the facts
		\begin{equation*}
			{\bf p}_i <_{\alpha} {\bf m}, \qquad {\bf q}_1 >_{\alpha} {\bf m}, \qquad {\bf q}_j <_{\alpha} {\bf m}, \quad \forall\,\, 1 \leq i \leq k,\; 2 \leq j \leq k.
		\end{equation*}
		Similarly, we can show that $\lambda_i=0$ for all $2 \leq i \leq k$. Since $k$ is arbitrary, this forces $\dim V_{\mu-{\bf e}_2} = \infty$, a contradiction.
		
		Again, assume that $d({\bf b})d({\bf b})v \neq 0$. Consider
		\begin{equation*}
			\begin{aligned}
				0&=\sum_{i=1}^{k} \lambda_i\, d({-{\bf{q}}_i})\,d({-{\bf{p}}_i})v=d( {\bf{m}}) \Big(\sum_{i=1}^{k} \lambda_i\, d({-{\bf{q}}_i})\,d({-{\bf{p}}_i})\Big)v\\
				&=\sum_{i=1}^{k} \lambda_i\,[d({\bf{m}}), d({-{\bf{q}}_i})\,d({-{\bf{p}}_i}) ]v=r\sum_{i=1}^{k} \lambda_i\,(3-2i)\, d \big((3-2i){\bf{b}}\big)\,d({-{\bf{p}}_i})v\\
				&=r\lambda_1 d({\bf{b}})\,d({\bf{b}})v +r \sum_{i=2}^{k} \lambda_i\,(3-2i)\, d \big((3-2i)\,{\bf{b}}\big)\,d \big((2i-1){\bf{b}}\big)v\\
				&=r\lambda_1 d({\bf{b}})\,d({\bf{b}})v.
			\end{aligned}
		\end{equation*}
		Hence $\lambda_1 = 0$, and similarly we can prove that $\lambda_2 = \cdots = \lambda_k = 0$. Finally, if $t^{\bf b}d({\bf b})v \neq 0$, then one can similarly show that $\{t^{-{\bf q}_i}d(-{\bf p}_i)v : 1 \leq i \leq k\}$ is linearly independent, and this completes the proof.
	\end{proof}
\end{lemma}	
	
\begin{theorem}\label{a36}
	Let $V$ be an irreducible GHW module over $\La$. Then $V$ is isomorphic to a highest weight module for $\La$.
	\begin{proof}
		Assume, for contradiction, that $V$ is not a highest weight module. Let $\mu \in P(V)$. By Lemma \ref{a35}, we can choose ${\bf a}=(a_1,a_2) \in \Z^2(+)$ such that $\mu+{\bf a} \in P(V)$. Since the order $<_{\alpha}$ is dense and ${\bf a} \in \Z^2(+)$, for any $k \in \N$ we may choose ${\bf b}=(r,s) \in \Z^2$ such that $(\bar{{\bf a}}|{\bf b}) \neq 0$ and $0 < r\alpha+s < \frac{a_1\alpha+a_2}{4k}$.
		It follows that $(0,0) <_{\alpha} i{\bf b} <_{\alpha} {\bf a}$, for all $1 \leq i \leq 4k$. Set
		\begin{equation}\label{a33}
			m_1= \text{max}\,\{j \in \Z: \mu+j{\bf{a}} \in P(V)\},\quad m_2= \text{max}\,\{j \in \Z: \mu+m_1{\bf{a}}+j {\bf{b}} \in P(V)\}
		\end{equation}
		Clearly, $m_1 \geq 1$ and $m_2 \geq 0$. Let $\gamma = \mu+m_1{\bf a}+m_2{\bf b}$,
		and choose $0 \neq v \in V_{\gamma}$. By \eqref{a33}, Lemma \ref{a32}, and the relations
		\[
		t^{\bf b}v = d({\bf b})v = t^{2{\bf b}}t^{-{\bf b}}v = d(2{\bf b})t^{-{\bf b}}v = t^{2{\bf b}}d(-{\bf b})v = d(2{\bf b})d(-{\bf b})v = 0,
		\]
		it follows that at least one of the vectors
		$t^{-{\bf b}}t^{-{\bf b}}v, \, t^{-{\bf b}}d(-{\bf b})v, \, d(-{\bf b})d(-{\bf b})v$
		is nonzero. In each case, we will show that $V_{\gamma-{\bf a}}$ is infinite-dimensional. For $1 \leq i \leq k$, set
		\[
		{\bf p}_i = (2i-1){\bf b} >_{\alpha} (0,0), \qquad
		{\bf q}_i = {\bf a}-{\bf p}_i >_{\alpha} (0,0), \qquad
		{\bf m} = {\bf a}-2{\bf b} >_{\alpha} (0,0).
		\]
		Assume first that $t^{-{\bf b}}t^{-{\bf b}}v \neq 0$.
		We now consider the set $\{t^{-{\bf{q}}_i}\,t^{-{\bf{p}}_i}v: 1 \leq i \leq k\} \subseteq V_{\gamma-{\bf{a}}}$ and claim that it is linearly independent. Indeed, if
		$\sum_{i=1}^{k} \lambda_i\, t^{-{\bf q}_i}t^{-{\bf p}_i}v = 0$, then applying $d({\bf m})$ yields
		\begin{equation}\label{a34}
			\begin{aligned}
				0&=d( {\bf{m}}) \Big(\sum_{i=1}^{k} \lambda_i\, t^{-{\bf{q}}_i}\,t^{-{\bf{p}}_i}\Big)v=\sum_{i=1}^{k} \lambda_i\,[\,d({\bf{m}}), \,t^{-{\bf{q}}_i}\,t^{-{\bf{p}}_i} ]\,v\\
				&=-\sum_{i=1}^{k} \lambda_i\,\Big( (\bar{\bf{m}}|{\bf{q}}_i)\,t^{{\bf{m}}-{\bf{q}}_i} t^{-{\bf{p}}_i} + (\bar{\bf{m}}|{\bf{p}}_i)\,t^{-{\bf{q}}_i}\,t^{{\bf{m}}-{\bf{p}}_i} \Big)\,v\\
			\end{aligned}
		\end{equation}
		Now if $t^{{\bf m}-{\bf p}_i}v \neq 0$ for some $i$, then $\gamma+{\bf m}-{\bf p}_i \in P(V)$ for that $i$. Since
		\[
		(\gamma+{\bf m}-{\bf p}_i)-(\gamma+{\bf b})={\bf a}-(2i+2){\bf b} \in \Z^2(+), \qquad \forall\, 1 \leq i \leq k,
		\]
		(S4) implies that $\gamma+{\bf b} \in P(V)$, contradicting the choice of $m_1$ and $m_2$. Hence $t^{{\bf m}-{\bf p}_i}v = 0$ for all $1 \leq i \leq k$. On the other hand,
		${\bf m}-{\bf q}_i = (2i-3){\bf b}$,
		so $t^{{\bf m}-{\bf q}_i}v = 0$ for all $2 \leq i \leq k$. Therefore \eqref{a34} reduces to
		\begin{equation*}
			0=-\lambda_1 (\bar{\bf{m}}|{\bf{q}}_1) t^{-{\bf{p}}_1}\,t^{{\bf{m}}-{\bf{q}}_1}v=-\lambda_1 (\bar{{\bf{a}}}|{\bf{b}})t^{-\bf{b}}t^{-\bf{b}}v.
		\end{equation*}
		Therefore, $\lambda_1 = 0$. Similarly, we can show that $\lambda_2=\cdots=\lambda_k=0$. Thus the above set is linearly independent. The other two cases can be handled similarly. In each case, for arbitrary $k \in \N$, we obtain a linearly independent subset of $V_{\gamma-{\bf a}}$ with $k$ elements. Consequently,
		$\dim V_{\gamma-{\bf a}} \geq k$ for all  $k \in \N$,
		which contradicts the fact that $V$ is a Harish-Chandra module. Therefore $V$ must be a highest weight module. 
	\end{proof}
\end{theorem}
\medskip
\medskip
Thus, Theorem \ref{a28} has been proved for the case $n=2$. Now we are in a position to prove Theorem \ref{a28} for arbitrary $n$. The proof is exactly analogous to that of \cite[Lemma 3.9]{lz}. For completeness and clarity, we summarize the key steps here. One can look at \cite[Lemma 3.9]{lz} for the detailed justifications of the steps given below.

\medskip
\medskip
\medskip
\noindent\textbf{Proof of Theorem \ref{a28}:}\label{a99} We prove the theorem by induction on $n$. The case $n=2$ is precisely Theorem \ref{a36}. Assume that the theorem holds for all $N \leq n-1$. We shall prove it for $N=n$.

Proceeding similarly as in \cite[Lemma 3.8]{lz}, we obtain
\begin{equation}\label{a41}
	P\big(\mathbb{L}(X,\boldsymbol{\beta}, M)\big)= P(X) \cup (\mu - \N \boldsymbol{\beta} + M).
\end{equation}
Here we use the corresponding rank $2$ divergence-zero algebra in place of the rank $2$ Virasoro algebra. Now, if there exist ${\bf m} \in \Z^n$ and a corank $1$ subgroup $M_0$ of $\Z^n$ such that $(\Lambda + {\bf m} + M_0) \cap P(V) \subseteq \{0\}$,
then the theorem follows from Proposition \ref{a40}. Hence we may assume that for every ${\bf m} \in \Z^n$ and every corank $1$ subgroup $M_0$ of $\Z^n$,
\[
(\Lambda + {\bf m} + M_0) \cap P(V) \nsubseteq \{0\}.
\]
Consider
$V_{\Lambda + {\bf m} + M_0} := \bigoplus_{{\bf m}_0 \in M_0} V_{\Lambda + {\bf m} + {\bf m}_0}$,
which is a Harish-Chandra module over a divergence-zero algebra of rank $(n-1)$ (instead of $\text{Vir}[G_0]$ in \cite[Lemma 3.8]{lz}). Applying Lemma 3.2, \eqref{a41}, and the inductive hypothesis, we deduce that for every ${\bf m} \in \Z^n$ and every corank $1$ subgroup $M_0$ of $\Z^n$, there exist a corank $1$ subgroup $M_{0,1}$ of $M_0$, ${\bf m}_{0,1} \in M_0 \setminus \{{\bf 0}\}$ such that
$M_0 = \Z {\bf m}_{0,1} \oplus M_{0,1}$,
and $\lambda_0' \in \Lambda + {\bf m} + M_0$ satisfying
\begin{equation}\label{a43}
	\lambda_0' + M_{0,1} - \N {\bf m}_{0,1} \subseteq P(V).
\end{equation}	
Next, we are going to prove that such a module does not exist under the assumption \eqref{a43}. Proceeding as in Claim 1 of \cite[Lemma 3.8]{lz}, we deduce that there do not exist $\lambda_0 \in P(V)$, $t_0 \in \Z$, ${\bf m}_0, {\bf m}_1 \in \Z^n \setminus \{{\bf 0}\}$, or subgroups $M_1' \subseteq M_0' \subseteq \Z^n$ with
$\Z^n = \Z {\bf m}_0 \oplus M_0'$, and $M_0' = \Z {\bf m}_1 \oplus M_1'$, such that
\begin{equation}\label{a42}
	\lambda_0 - \Z_+ {\bf m}_1 + M_1', \,\,
	\lambda_0 + t_0 {\bf m}_1 + \Z_+ {\bf m}_1 + M_1' \subseteq P(V),
\end{equation}
and, if $t_0 \leq 0$, then $\lambda_0 + M_0' \subseteq P(V)$.

For each $t \in \Z$, set $\overline{M}_t=t{\bf{e}}_1 \oplus \Z {\bf{e}}_2 \oplus \cdots \oplus \Z {\bf{e}}_n$. Again as in Claim 2 of \cite[Lemma 3.8]{lz}, we can prove that if $\lambda_0 \in \Lambda+\Z^n$, ${\bf m}_1, {\bf m}_1' \in \overline{M}_0$, and $M_1, M_1'$ are subgroups of $\overline{M}_0$ satisfying
$\overline{M}_0 = \Z {\bf m}_1 \oplus M_1 = \Z {\bf m}_1' \oplus M_1'$, and
\[
\lambda_0 - \N {\bf m}_1 + M_1, \,\,
\lambda_0 - \N {\bf m}_1' + M_1' \subseteq P(V),
\]
then $M_1 = M_1'$.

Finally, set
$V_{\lambda+\overline{M}_t}:=\bigoplus_{{\bf m} \in \overline{M}_t} V_{\lambda+{\bf m}}$,
which is again a Harish-Chandra module over a divergence-zero algebra of rank $(n-1)$.
Then, exactly as in (3.42) of \cite[Lemma 3.8]{lz}, we can deduce that there exist a corank $1$ subgroup $M_0$ of $\overline{M}_0$, $\alpha_t \in \Lambda+\overline{M}_t$, and ${\bf m}_0 \in \overline{M}_0$ with
$\overline{M}_0 = \Z {\bf m}_0 \oplus M_0$ such that either
\begin{equation*}
	P(V_{\Lambda+\overline{M}_t}) \setminus \{{\bf{0}}\}= (\al_t + \Z_+ {\bf{m}}_0 + M_0) \setminus \{{\bf{0}}\} \,\,\,\,\text{or}\,\,\,\, (\al_t + \Z_- {\bf{m}}_0 + M_0) \setminus \{{\bf{0}}\}.
\end{equation*}

Proceeding exactly as in \cite[Lemma 3.8]{lz}, we finally obtain a contradiction to \eqref{a42}. This completes the proof.

\section{Main Theorem}\label{a97}
	We now summarize the results obtained so far. By combining Theorem \ref{a28}, Theorem \ref{a75}, and Theorem \ref{a90}, we obtain a classification of irreducible Harish-Chandra modules over $\D \ltimes \mathcal{A}_n$ with nontrivial $\mathcal{A}_n'$-action.
	\begin{theorem}\label{a100}
		Let $V$ be an irreducible Harish-Chandra module over $\D \ltimes \mathcal{A}_n$ with nontrivial $\mathcal{A}_n'$-action. Then:
		\begin{enumerate}
			\item $V$ is either cuspidal or a generalized highest weight module.
			\vspace{0.1cm}
			\item If $V$ is cuspidal, then $V$ is isomorphic to $V(\lambda, \boldsymbol{\alpha}, c)$ for some dominant integral weight $\lambda$ of $\mathfrak{sl}_n$, $\boldsymbol{\alpha} \in \C^n$, and $c \in \C^{*}$.
			\vspace{0.1cm}
			\item If $V$ is a generalized highest weight module, then $V$ is isomorphic to $\mathbb{L}(X,\boldsymbol{\beta}, M)$ for some subgroup $M$ of $\Z^n$, $\boldsymbol{\beta} \in \Z^n$ such that $\Z^n= M \oplus \Z \boldsymbol{\beta}$, and some irreducible $\EuScript{G}_M$-module $X$.
		\end{enumerate}
	\end{theorem}

	\subsection*{Funding:} The author gratefully acknowledges support from the National Board for Higher Mathematics through a post-doctoral fellowship (Ref. No. 0204/27/(34)/2023/R\&D-II/11935.
	\subsection*{Acknowledgments.} The author would like to thank Prof. Sachin Sharma for suggesting the problem and for several helpful discussions. The author also thanks Ms. Ananya Gaur for many helpful discussions.
	
	%\subsection*{Data Availability:} Data sharing is not applicable to this article as no datasets were generated or analyzed during the current study.

\end{document}